\documentclass[12pt, reqno]{amsart}
\usepackage{amsmath, amsthm, amscd, amsfonts, amssymb, graphicx, color, mathrsfs}
\usepackage[bookmarksnumbered, colorlinks, plainpages]{hyperref}
\usepackage[all]{xy} 
\usepackage{slashed}

\newtheorem{theorem}{Theorem}[section]
\newtheorem{lemma}[theorem]{Lemma}

\newtheorem{proposition}[theorem]{Proposition}
\newtheorem{corollary}[theorem]{Corollary}

\theoremstyle{definition}
\newtheorem{definition}[theorem]{Definition}

\newtheorem{assumption}[theorem]{Assumption}

\theoremstyle{remark}
\newtheorem{remark}[theorem]{Remark}
\numberwithin{equation}{section}

\begin{document}
\setcounter{page}{1}

\title[Spectral inequalities on compact manifolds]{Spectral inequalities for elliptic pseudo-differential operators on closed manifolds}

  \author[D. Cardona]{Duv\'an Cardona}
\address{
  Duv\'an Cardona:
  \endgraf
  Department of Mathematics: Analysis, Logic and Discrete Mathematics
   \endgraf
  Ghent University.
  \endgraf
  Krijgslaan 281, Building S8, 9000 Ghent, Belgium.
  \endgraf
  {\it E-mail address} {\rm duvanc306@gmail.com, duvan.cardonasanchez@ugent.be}
  }

\thanks{This project has received funding from the European Research Council (ERC) under the European Union’s Horizon 2020 research and innovation programme (grant agreement NO: 694126-DyCon).   This work also has been supported by  the FWO  Odysseus  1  grant  G.0H94.18N:  Analysis  and  Partial Differential Equations and by the Methusalem programme of the Ghent University Special Research Fund (BOF)
(Grant number 01M01021).
}

\begin{abstract} Let  $(M,g)$ be a closed Riemannian manifold. The aim of this work is to prove  the Lebeau-Robbiano spectral inequality for a positive elliptic pseudo-differential operator $E(x,D)$ on $M,$ of order $\nu>0,$  in the H\"ormander class $\Psi^\nu_{\rho,\delta}(M).$  In control theory this has been an open problem prior to this work.   As an application of this fundamental result, we establish the null-controllability of the (fractional) heat equation associated with  $E(x,D).$ The  sensor  $\omega\subset M$ in the observability inequality is an open subset of $M.$ The obtained results (that are, the corresponding spectral inequality for an elliptic operator and the null-controllability  for  its diffusion model) extend in the setting of closed manifolds,  classical results of the control theory, as the spectral inequality due  to Lebeau and Robbiano and their result on the  null-controllability of the heat equation giving a complete picture of the subject in the setting of closed manifolds. As a consequence of our analysis we prove the Donnelly-Fefferman  inequality for positive pseudo-differential operators for sums of eigenfunctions.   For the proof of the spectral inequality we introduce a periodization approach in time inspired by the global pseudo-differential calculus due to Ruzhansky and Turunen.
\end{abstract} 
\maketitle
\tableofcontents
\allowdisplaybreaks
\section{Introduction}

In 1995  Labeau and Robbiano \cite{LabeauRobbiano1995} proved the null-controllability of the heat equation for the Laplacian on a compact manifold (with or without smooth boundary). An instrumental tool in the proof of Labeu-Robbiano's result
 is a spectral inequality for the wave packets of eigenfunctions of the Laplacian originated from  \cite{LabeauRobbiano1995,JerisonLabeau}. In the case of a single eigenfuction $\rho_j,$ the Lebeau-Robbiano spectral inequality implies the doubling property proved by Donnelly and Fefferman \cite{DonnellyFefferman}. The spectral inequality from \cite{LabeauRobbiano1995,JerisonLabeau} has been used in many areas of analysis, for instance, in the aforementioned result of the control theory, or as in \cite{DonnellyFefferman} in the geometric analysis of  the Hausdorff measure of nodal sets.  In this work we extend such a spectral inequality to the case of a positive elliptic pseudo-differential operator $E(x,D)$ of arbitrary  order $\nu>0$ and we apply this spectral  inequality to the null-controllability of the heat equation (and also of the fractional heat equation) associated to $E(x,D).$ 
 
 The analysis here includes the family of elliptic differential operators of arbitrary order with smooth coefficients  where the main obstruction in proving the Lebeau-Robbiano spectral inequality is the lack of Carleman estimates. Indeed, in the high-dimensional case,  the so-called H\"ormander subellipticity
condition does not hold, which is typical for product operators. Our approach compensates it with the use of the calculus of (possibly non-local) pseudo-differential operators on compact manifolds due to H\"ormander \cite{Hormander1985III}. Moreover, we prove that any positive elliptic pseudo-differential operator of positive order in the standard $(\rho,\delta)$-H\"ormander class on a closed manifold $M$ satisfies the Lebeau-Robbiano spectral inequality.

In general pseudo-differential operators are non-local and the use of Carleman estimates, which is the analytical tool by excellence in the proof of the Lebeau-Robbiano spectral inequality and their subsequent generalisations, are not valid. This work introduces a new approach in the proof of these spectral inequalities for  elliptic differential operators, embedding this problem in the general setting  of the pseudo-differential calculus due to H\"ormander where we exploit e.g. the Calder\'on-Vaillancourt theorem and other microlocal properties of the pseudo-differential calculus.

The validity of the Lebeau-Robbiano spectral inequality for an elliptic pseudo-differential operator implies the null-controllability for its corresponding diffusion model. This is a very well known consequence  of the equivalence between the null-controllability of a linear model and its corresponding observability inequality due to the Hilbert Uniqueness Method developed by J. -L. Lions \cite{JLLions}. The null-controllability for the diffusion model associated to the operator $E(x,D)^\alpha$ as well as its corresponding spectral inequality are presented in Theorem \ref{Main:theorem} and in Theorem \ref{The:Spectral:Inequality} below. 
\subsection{Outline and main results} To discuss the contributions of this work we precise some notations.  
Let $(M,g)$ be a compact Riemannian manifold without boundary and of dimension $n$.  Let $\omega$ be a non-empty open subset of $M$ and let $1_\omega$ be its characteristic function. For any $\nu>0,$ let $E(x,D)$ be a positive elliptic pseudo-differential operator of order $\nu$ in the H\"ormander class $\Psi^\nu_{\rho,\delta}(M),$ see  \cite{Hormander1985III}. The conditions $\rho\geq 1-\delta$ and  $0\leq \delta<\rho\leq 1,$ imply that  the class $\Psi^\nu_{\rho,\delta}(M)$ is invariant under changes of coordinates and in particular, for $\nu\in 2\mathbb{N},$ the family of elliptic operators in the class $\Psi^\nu_{1,0}(M)$ contains any positive elliptic differential operator with smooth coefficients of order $\nu=2m.$ One of the goals of this work is to establish the null controllability for the model of fractional diffusion
\begin{equation}\label{Main}
\begin{cases}u_t(x,t)+ E(x,D)^\alpha u(x,t)=g(x,t)\cdot 1_\omega (x) ,& (x,t)\in M\times (0,T),
\\u(0,x)=u_0,\end{cases}
\end{equation} at any time $T>0,$ that is, to investigate the conditions on  the pair $(\alpha,\nu),$ in order that for any initial condition $u_0,$ there is an input function $g\in L^2(M\times (0,T))$ such that the solution to \eqref{Main} vanishes in time $T,$ that is $u(x,T)=0,$ $x\in M.$ 

Going to the historical aspects of the control theory for the heat equation, in the context of compact manifolds, it started with the aforementioned result  by Lebeau and Robbiano \cite{LabeauRobbiano1995} where they proved the null-controllability of \eqref{Main} in the case of the Laplacian $E(x,D)=-\Delta_g,$ and for   $\alpha=1$ (even if the manifold has a smooth non-empty boundary).  Then the general result by Miller \cite{Miller2006} proved in the setting of separable Hilbert spaces allows to use the spectral inequality in \cite{LabeauRobbiano1995} to get the same result for $E(x,D)=(-\Delta_g)^{\alpha}$ for all $\alpha>1/2.$ Moreover, the sharpness of the inequality $\alpha>1/2$ was analysed in  \cite{Lu2013}. 

One of the main tools in proving the null-controllability result in the case of the Laplacian is  the following spectral  inequality for wave packets of eigenfunctions, see Jerison and Lebeau \cite{JerisonLabeau},  and \cite{LabeauRobbiano1995,LebeauLebeau1998}.
\begin{theorem}
    Let $M$ be a compact Riemmanian manifold with (or without) smooth boundary $\partial M.$ Let $(\rho_j,\lambda_j^2)$ be the corresponding spectral data of the Laplacian $-\Delta_g,$ determined by the eigenvalue problem $-\Delta_g\rho_j=\lambda_j^2\rho_j.$ 
    
    Then, for any non-empty open subset $\omega\subset M,$ we have the loss of orthogonality estimate
\begin{equation}\label{Observation:Estimte}
    \Vert \varkappa\Vert_{L^2(M)}\leq C_1e^{C_2 {\lambda}}\Vert \varkappa\Vert_{L^2(\omega)},\,\,\,\varkappa\in \textnormal{span}\{\rho_j:\lambda_j\leq \lambda\}.
\end{equation}Moreover, the growth constant $e^{C_2 {\lambda}}$ is sharp.
\end{theorem}

Observation estimates of the type \eqref{Observation:Estimte} have many applications in control theory and in geometric analysis. For instance, in control theory only to mention a few, it was used in the aforementioned result of null-controllability of the heat equation \cite{LabeauRobbiano1995}, in the  null-controllability of the thermoelasticity system \cite{LebeauLebeau1998}, the
null-controllability of the thermoelastic plate system \cite{BenabdallahNaso2002, Miller2007},  the null-controllability for the elasticity model of the bi-Laplacian $\Delta_g^2$ (endowed with ``clampled'' boundary conditions) \cite{RousseauRobbiano2020}, and the null-controllability of some systems of parabolic PDE \cite{Leautaud2010}. On the other hand, in geometric analysis it can be used to estimate the Hausdorff measure of   the nodal set of finite sums of eigenfunctions of $-\Delta_g,$ when the manifold is  analytic, see \cite{Lin1991,JerisonLabeau} generalizing a result by Donnelly and Fefferman \cite{DonnellyFefferman}. The Lebeau-Robbiano spectral inequality has shown to be valid even when $\omega$ is a measurable set and satisfies some additional geometric conditions, see e.g. \cite{Apraiz} and references therein.

Let us consider a positive elliptic pseudo-differential operator $A\in \Psi^\nu_{\rho,\delta}(M)$ with positive order $\nu>0$ on a compact Riemannian manifold $(M,g).$  The  sequence $\{(\mu_j=\lambda_j^\nu,\rho_j):j\in \mathbb{N}\}$ determines the spectral data of the operator $E(x,D)$ that are determined by the eigenvalue problem: \begin{center}
      $E(x,D)\rho_j=\lambda_j^\nu \rho_j,$ $0\leq \lambda_j\leq\lambda_{j+1}\rightarrow \infty.$ 
  \end{center} 
The contributions of this work start with the following  null-controllability statement  of the fractional heat equation for elliptic operators on compact manifolds.  Here, $0\leq \delta<\rho\leq 1,$ and  $\rho\geq 1-\delta,$ are the required conditions in order that the classes $\Psi^\nu_{\rho,\delta}(M)$ became invariant under changes of local coordinate systems.
\begin{theorem}\label{Main:theorem} Let $E(x,D)\in \Psi^\nu_{\rho,\delta}(M)$ be a positive elliptic pseudo-differential operator of order $\nu>0.$ Let $(x,\xi)\in T^*M,$ and assume that for any $\xi\neq 0,$ $E(x,\xi)>0$ is strictly positive.   Then,
for any $\alpha>1/\nu,$  the fractional diffusion problem in \eqref{Main} is null-controllable at any time $T>0.$
\end{theorem}
As a consequence of Theorem \ref{Main:theorem} we obtain the null-controllability of the fractional diffusion model \eqref{Main} for a class of elliptic differential operators of arbitrary (even) order and with smooth coefficients.
 \begin{corollary}Let $\nu\in 2\mathbb{N},$ and let $E(x,D)$ be a positive elliptic differential operator of order $\nu=2m$ on $M$ with smooth coefficients. Let $(x,\xi)\in T^*M,$ and assume that for any $\xi\neq 0,$ $E(x,\xi)>0$ is strictly positive.  Then,
for any $\alpha>1/\nu,$  the fractional diffusion problem in \eqref{Main} is null-controllable at any time $T>0.$ 
 \end{corollary}
By following the strategy of Lebeau and Robbiano \cite{LabeauRobbiano1995}, we we shall derive the proof of the controllability Theorem \ref{Main:theorem}, from the corresponding spectral inequality for the elliptic pseudo-differential operator $E(x,D).$ 

\begin{theorem}\label{The:Spectral:Inequality} Let $\nu>0,$ and let $0\leq \delta<\rho\leq 1$ be such that $\rho\geq 1-\delta.$ Let $E(x,D)\in \Psi^\nu_{\rho,\delta}(M)$ be an elliptic positive pseudo-differential operator of order $\nu>0.$ Let $(x,\xi)\in T^*M,$ and assume that for any $\xi\neq 0,$ $E(x,\xi)>0$ is strictly positive.  Then, for any non-empty  open subset $\omega\subset M,$ we have
\begin{equation}\label{Spectral:Inequality:Intro}
    \Vert \varkappa\Vert_{L^2(M)}\leq C_1e^{C_2 {\lambda}}\Vert \varkappa\Vert_{L^2(\omega)},\,\,\,\varkappa\in \textnormal{span}\{\rho_j:\lambda_j\leq \lambda\}.
\end{equation}
\end{theorem}
As a consequence of Theorem \ref{The:Spectral:Inequality} we obtain in the case of elliptic differential operators the following spectral inequality.

\begin{corollary}\label{Spectral:ineq:Diff} Let $\nu\in 2\mathbb{N},$ and let $E(x,D)$ be a positive elliptic differential operator  of order $\nu$ and with smooth coefficients on $M.$ Let $(x,\xi)\in T^*M,$ and assume that for any  $\xi\neq 0,$ $E(x,\xi)>0$ is strictly positive.  Then, for any non-empty  open subset $\omega\subset M,$ we have
\begin{equation}\label{Lebeau:Robbiano:Diff}
    \Vert \varkappa\Vert_{L^2(M)}\leq C_1e^{C_2 {\lambda}}\Vert \varkappa\Vert_{L^2(\omega)},\,\,\,\varkappa\in \textnormal{span}\{\rho_j:\lambda_j\leq \lambda\}.
\end{equation}
 \end{corollary}

 The point of view that we adopt, and that is one of the novelties of this work, is the use of the pseudo-differential calculus on compact manifolds without boundary (see H\"ormander \cite{Hormander1985III}) in the proof of the spectral inequality \eqref{Spectral:Inequality:Intro}. One of the reasons for introducing this technique comes from the lack of Carleman estimates for pseudo-differential operators (which in general are not local operators)  (see e.g. the discusssion by Rousseau and Robbiano  in \cite[Subsection 1.4]{RousseauRobbiano2020}). Nevertheless, when using the theory of pseudo-differential operators, other techniques of the modern microlocal analysis as the Calder\'on-Vaillancourt theorem for closed manifolds, and  the analytic functional calculus of elliptic operators can be applied in order to prove an interpolation inequality of the form $\Vert F\Vert_{H^1(M\times (\alpha,T-\alpha))}\lesssim \Vert F\Vert_{H^1(M\times (0,T)}^\kappa \Vert F \Vert^{1-\kappa}_{L^2(\omega)}$ for a suitable wave function $F$ and that is a crucial tool in the proof of  \eqref{Spectral:Inequality:Intro}. One of the difficulties in the proof of this inequality is that the operators $E(x,D)$ could be in general of  non-local type.

We also observe that the main reason in considering only in our analysis the case of a compact manifold without boundary is the use of the pseudo-differential calculus by itself. Indeed, when the manifold $M$ has a non-empty smooth boundary $\partial M$, the interior $\textnormal{int}(M)$ of $M$ is an open submanifold, and the  pseudo-differential calculus for  $M$ requires a more delicate analysis. In this setting, pseudo-differential operators are matrices acting from $C^\infty(\textnormal{int}(M))\oplus C^\infty(\partial M)$ into itself and the calculus uses the transmission property introduced by H\"ormander. We refer the reader to the seminal work of Boutet de Monvel in  \cite{Boutet:Acta} for details.

Now, we briefly discuss our main results in relation with some known results.
\begin{remark}
In the case where $E(x,D)=-\Delta_g$ is the positive Laplacian, the symbol $E(x,\xi)=\Vert \xi\Vert_{g}>0$ is strictly positive on $T^*M\setminus \{0\}.$ So, our Theorem \ref{The:Spectral:Inequality} recovers the Lebeau-Robianno inequality and then Theorem \ref{Main:theorem} recovers the null controllability of the heat equation for the Laplacian. We observe that in the case of fractional powers $E(x,D)=(-\Delta_g)^{\alpha},$ of order $\nu=2\alpha,$ $\alpha>0,$ the heat equation for the fractional operator becomes to be null-controllable in any time $T>0,$ always that $\alpha>1/2.$ This recover the fractional result due to Miller \cite{Miller2006}.  
\end{remark}
\begin{remark}
    If $E(x,D)=-\sum_{i,j=1}^na_{ij}(x)\partial^2_{x_ix_j}+\textnormal{lower terms},$ is a positive elliptic second order differential operator with real smooth coefficients $a_{ij}\in C^\infty(M),$ and $a_{ij}=a_{ji}$ for any $(i,j),$ the ellipticity condition implies that the principal symbol 
   $
        E(x,\xi)=\sum_{i,j=1}^{n}a_{ij}(x)\xi_i\xi_j\geq C\Vert\xi \Vert_{g}^2
   $ is strictly positive on $T^*M\setminus \{0\}.$ In this setting, the positivity condition on $E(x,D)$ and on its symbol $E(x,\xi)$ implies the null-controllability of \eqref{Main} for all $\alpha>1/2.$ This fact follows from  the validity of the spectral inequality for $E(x,D),$ see e.g. \cite{RousseauLebeau2012} and Theorem 1 in Miller \cite[Page 263]{Miller2006} (see Theorem \ref{Miller:Theorem} below). This situation is also covered by our main results.
\end{remark}
\begin{remark}
If $E(x,D)=\sum_{|\alpha|\leq \nu}a_{\alpha}(x)D_x^\alpha$ is a positive elliptic differential operator with smooth coefficients, necessarily $\nu\in 2\mathbb{N},$ and Corollary \ref{Spectral:ineq:Diff} gives the Lebeau-Robbiano spectral inequality \eqref{Lebeau:Robbiano:Diff}. Different to the case $\nu=2,$ where the spectral inequality \eqref{Lebeau:Robbiano:Diff} can be derived by the use of Carleman inequalities, the case of higher-order differential operators (that is when $\nu\geq 4$) there is a lack of Carleman estimates since the H\"ormander subellipticity condition for the weight function could be not valid. This is the case of product operators $E(x,D)=Q_1Q_2,$ see Rousseau and Robbiano  \cite[Subsection 1.4]{RousseauRobbiano2020}). Recently, for the case of the bi-Laplacian $-\Delta_g^2,$ the Lebeau-Robianno spectral inequality was proved  in \cite{RousseauRobbiano2020} endowed of the so called {\it clamped boundary conditions} leaving open the question for similar boundary conditions for higher-dimensional powers $(-\Delta_g)^k,$ $k\geq 3,$ on a manifold with non-empty smooth boundary and with the operator endowed with the natural boundary conditions, e.g. $u|_{\partial M}=\cdots= \partial_\nu^{k-1}u|_{\partial M}=0,$ or more general Lopatinskii-type conditions.
\end{remark}
\begin{remark}
    The Lebeau-Robbiano spectral inequality remains valid even if the coefficients of a second-order elliptic differential operators are not smooth. One allows the case of Lipchitz coefficients, coefficients with jumps at an interface, etc. We refer the reader to  \cite[Pages 55-56]{FuLuZhang2020} and the references discussed there for details.
\end{remark}
\begin{remark}
In the case of a single eigenfuction $\rho_j,$ the estimate in \eqref{Observation:Estimte} can be derived from the doubling property proved by Donnelly and Fefferman \cite{DonnellyFefferman}, see also \cite{Lin1991,JerisonLabeau}.  The sharpness of the estimate \eqref{Observation:Estimte} was proved by Jerison and Lebeau \cite{JerisonLabeau}. 
\end{remark}
In the general case of a positive pseudo-differential operator we have the following Donnelly-Fefferman type inequality (extending the previous result in \cite{DonnellyFefferman}). The proof of Theorem \ref{Donnelly-Type:Fefferman:Inequality} will be given in Section \ref{Donnelly:Fefferman:Inequality}.
\begin{theorem}\label{Donnelly-Type:Fefferman:Inequality} Let $\nu>0,$ and let $0\leq \delta<\rho\leq 1$ be such that $\rho\geq 1-\delta.$ Let $E(x,D)\in \Psi^\nu_{\rho,\delta}(M)$ be an elliptic positive pseudo-differential operator of order $\nu>0.$ Let $(x,\xi)\in T^*M,$ and assume that for any $\xi\neq 0,$ $E(x,\xi)>0$ is strictly positive. 
    For any $R>0$ let $B(x,R)$ be a ball defined by the geodesic distance, of radius $R>0$ and centred at $x.$ Then,
\begin{equation}\label{Donnelly-Fefferman}
    \sup_{B(x,2R)}|\varkappa|\leq  e^{C_1' {\lambda}+C_2'} \sup_{B(x,R)}|\varkappa|,\,\,\,\varkappa\in \textnormal{span}\{\rho_j:\lambda_j\leq \lambda\},
\end{equation}with $C_1'=C_{1}'(R)$ and $C_2'=C_2'(R)$ are dependent  only on the radius $R>0$ but  not on $\varkappa.$
\end{theorem}

\begin{remark}
We observe that  the case of the 1D spectral fractional heat equation was analysed in \cite{MicuZuazua2006}. Moreover, the analysis of null-controllability for the heat equation for the  1D-integral fractional Laplacian has been considered in \cite{Biccari}.    
\end{remark}

\subsection{Structure of the work}
This article is organised as follows. In Section \ref{Preliminaries} we present the basics on pseudo-differential operators and the null-controllability results for fractional differential problems on Hilbert spaces. In particular we present the basics about the construction of complex powers of pseudo-differential operators. Then, the spectral inequality in Theorem \ref{The:Spectral:Inequality} and the null-controllability result in Theorem \ref{Main:theorem}  are proved in Section \ref{SectionofProofs}.

\section{Null-controllability of the fractional heat equation and observability inequalities}\label{SectionofProofs}

Let us consider an orientable compact Riemannian manifold $(M,g)$ (without boundary) of dimension $n.$ Let us consider:
\begin{itemize}
    \item[-] The $L^2$-space $L^2(M)=L^2(M,d\textnormal{Vol}_g)$  associated with the volume form $d\textnormal{Vol}_g,$ for the metric $g.$ For simplicity we always write $dx:=d\textnormal{Vol}_g(x).$
    \item[-] The inner product $(f,g)=\smallint_{M}f(x)\overline{g(x)}dx$  on $L^2(M).$
    \item[-] For any pair $(\rho,\delta)\in [0,1]^2,$ such that $0\leq \delta<\rho\leq 1$ and $\rho\geq 1-\delta,$ we consider the class $\Psi^m_{\rho,\delta}(M)$ of pseudo-differential operators of order $m\in \mathbb{R}.$ We refer the reader to  Appendix I in Section \ref{Preliminaries} for the definition of these classes.  
\end{itemize}

If  $E(x,D)\in \Psi^\nu_{\rho,\delta}(M)$ of positive order $\nu>0,$ in this section we establish the null controllability for the model of fractional diffusion
\begin{equation}\label{Main:Eq:22:3}
\begin{cases}u_t(x,t)+ E(x,D)^\alpha u(x,t)=g(x,t)\cdot 1_\omega (x) ,& (x,t)\in M\times (0,T),
\\u(0,x)=u_0,\end{cases}
\end{equation} at any time $T>0,$ that is, we will prove that the solution to \eqref{Main:Eq:22:3} satisfies  $u(x,T)=0,$ if  $\alpha>1/\nu.$ As described in the introduction, first, we require validity of the spectral inequality \eqref{Spectral:Inequality:Intro}. We employ the following notation.
\begin{itemize}
    \item[-] We denote by $\omega\neq \emptyset$  an open  subset in $M,$ as well as a generic compact subset in $M$ will be denoted by $K.$  
    \item[-] For any $\lambda>0,$ we denote by $\textnormal{E}_{\lambda}$ the spectral projection on $L^2(M)$ defined by
    \begin{equation}
        \textnormal{E}_{\lambda}f=\sum_{\lambda_j\leq \lambda}(f,\rho_j)\rho_j,\,\, f\in L^2(M),
    \end{equation}where $\{\mu_j:=\lambda_j^\nu,\rho_j\}$ are the corresponding spectral data $$ E(x,D)\rho_j=\lambda_j^\nu \rho_j,\,\,\lambda_j\geq 0,$$ with the eigenfunctions $ \rho_j$ being $L^2$-normalised. 
    \item[-] For any $T>0,$ let us consider the space-time manifold $$M_T:=M\times (0,T),$$ and the Sobolev space $H^s(M_T)$ of order $s\in \mathbb{N},$ which consists of  all distribution
   $f$ on $M$ such that,
      \begin{equation}  
      \Vert f\Vert_{H^s(M_T)}^2=\sum_{0\leq j\leq s}\smallint\limits_0^T\smallint\limits_M\left[\left|\partial_t^jf(x,t)\right|^2+|(1-\Delta_g)^{\frac{j}{2}}f(x,t)|^2\right]dxdt<\infty,
    \end{equation}where $-\Delta_g$ is the positive Laplacian associated to the metric $g.$
    \item[-]  
Observe that under the identification $-(T+\varepsilon)\sim T+\varepsilon,$ $\varepsilon>0,$ the manifold $[-(T+\varepsilon),T+\varepsilon)$ can be identified with the dilated torus $\mathbb{T}_{T,\varepsilon}:=\mathbb{R}/2(T+\varepsilon)\mathbb{Z}\cong [-(T+\varepsilon),T+\varepsilon)$ and the resulting manifold is diffeomorphic to the circle $\mathbb{S}^1$, and then, it is compact and without boundary. Note also, that the operator $-\partial_t^2$ became, up to a constant, the (positive) Laplace-Beltrami operator on $\mathbb{T}_{T,\varepsilon}.$ Finally,
note that $-\partial_t^2\in \Psi^2_{1,0}(\mathbb{T}_{T,\varepsilon})$ is a positive and elliptic differential operator of second order on $\mathbb{T}_{T,\varepsilon}$ and then  $$E(x,t,D,\partial_t)=-\partial_t^2+E(x,D)^{\frac{2}{\nu}} \in \Psi^2_{1,\delta}({M}\times \mathbb{T}_{T,\varepsilon})$$ is also a positive and elliptic pseudo-differential operator on the product manifold ${M}\times \mathbb{T}_{T,\varepsilon}.$ Note that $M_T$ can be viewed as  an open sub-manifold of ${M}\times \mathbb{T}_{T,\varepsilon}.$
\end{itemize}
Our analysis for the null-controllability of the heat equation starts with the   spectral inequality in
Proposition \ref{Lemma:LR;Ineq} below that corresponds in Theorem \ref{The:Spectral:Inequality} to the case where $E(x,D)$ satisfies the lower bound $E(x,D)\geq cI,$ for some $c>0.$ Then, the proof of Theorem \ref{The:Spectral:Inequality} will be deduced from this particular situation.

\begin{proposition}\label{Lemma:LR;Ineq}  Let $(M,g)$ be a compact Riemannian manifold without boundary. Let   $0\leq  \delta<\rho\leq 1,$ be such  that $1-\delta\geq \rho$. Let $E(x,D)\in \Psi^\nu_{\rho,\delta}(M) $ be an elliptic positive pseudo-differential operator of order $\nu>0,$ and assume that for some $c>0,$  $E(x,D)\geq cI$ in $L^2(M)$ and that for any $(x,\xi)\in T^*M,$ $E(x,\xi)\geq c$ is strictly positive.

 Then, for any open subset $\omega\subset M,$   any $a_j\in \mathbb{R},$ and all $\lambda>0,$ the following observability inequality holds 
\begin{equation}\label{ObservabilityInequality}
    \left(\sum_{\lambda_j\leq \lambda}a_j^2\right)^\frac{1}{2}\leq C_1e^{C_2{\lambda}}\left\Vert \sum_{\lambda_j\leq \lambda}a_j\rho_j(x)  \right\Vert_{L^2(\omega)}.
\end{equation}
\end{proposition}
We postpone the proof of Proposition \ref{Lemma:LR;Ineq} for a moment. Indeed, for the proof of this statement we require a delicate microlocal analysis. First, we recall an interpolation inequality for the Laplacian in the next subsection.

\subsection{An interpolation inequality}
For our further analysis we require the following interpolation inequality.
\begin{lemma}\label{CarlemanInequality} Let us consider the operator $A(x,t,D,\partial_t)=-\partial_t^2-\Delta_g\in \Psi^2_{1,0}(M_T)$ with  $M_T=M\times (0,T)$. Let $\omega$ be a non-empty open subset  in $M.$ 

Then, for any $T>0$ and all $\alpha\in (0,T/2),$ there exists $\delta\in (0,1)$ such that
\begin{equation}\label{Global:Carleman:Estimate}
    \Vert \phi \Vert_{H^1(M\times (\alpha,T-\alpha))}\leq C\Vert \phi\Vert_{H^1(M_T)}^\delta\left(\Vert A(x,t,D,\partial_t)\phi \Vert_{L^2(M\times (0,T))} +\Vert \partial_t\phi\Vert_{L^2(\omega)} \right)^{1-\delta},
\end{equation}for all $\phi\in H^2(M\times (0,T))$ such that $\phi=0$ in $M\times \{0\}.$
\end{lemma} The inequality  \eqref{Global:Carleman:Estimate} in Lemma \ref{Global:Carleman:Estimate} has been obtained in \cite{LebeauLebeau1998} in the case of a compact manifold without boundary.  We also refer the reader to Rousseau and Lebeau \cite{RousseauLebeau2012} for alternative proofs of this result even, for second-order elliptic operators with smooth coefficients.

\subsection{Proof of the spectral inequality}\label{Begininnig} 

For any $\lambda>0,$ let $\varkappa
\in \textnormal{Im}(\textnormal{E}_{\lambda}),$ that is $\varkappa$ can be written as
\begin{equation}
    \varkappa(x)=\sum_{\lambda_j\leq \lambda}a_j\rho_j(x).
\end{equation}

Observe that for the proof of the spectral inequality \eqref{ObservabilityInequality}, is enough to prove that the function
\begin{equation}\label{The:main:inequality:here}
    F(x,t):=\sum_{\lambda_j\leq \lambda}\frac{\sinh(\lambda_jt)}{\lambda_j} a_j\rho_j(x),\,\,(x,t)\in M_T:=M\times [0,T],
\end{equation}satisfies the interpolation inequality
\begin{equation}\label{Interpolation:Inequality}
    \Vert F \Vert_{H^1(M\times (\alpha,T-\alpha))}\leq C\Vert F\Vert_{H^1(M_{T})}^\kappa \Vert \varkappa \Vert_{L^2(\omega)}^{1-\kappa}.
\end{equation}

Indeed, one can follows e.g. the analysis in \cite{LebeauLebeau1998}. It starts  using the Parseval theorem  in order to have
\begin{align*} 
\Vert F\Vert^2_{H^1(M\times (\alpha,T-\alpha))} &\geq C \Vert F\Vert^2_{L^2(M\times (\alpha,T-\alpha))}\\
&=\smallint_\alpha^{T-\alpha}\smallint\limits_M\left|\sum_{\lambda_j\leq \lambda}\frac{\sinh(\lambda_jt)}{\lambda_j} a_j\rho_j(x)\right|^2dx\,dt\\
&=\sum_{\lambda_j\leq \lambda}|a_j|^2\smallint_\alpha^{T-\alpha}\left|\frac{\sinh(\lambda_jt)}{\lambda_j}\right|^2dt\\
&\geq \sum_{\lambda_j\leq C\lambda}|a_j|^2\smallint_\alpha^{T-\alpha}t^2dt\\
&=C_{\alpha}\sum_{\lambda_j\leq \lambda}|a_j|^2.
\end{align*}
Observing that
\begin{equation}
   \partial_tF(x,0)= \sum_{\lambda_j\leq \lambda}a_j\rho_j(x),\textnormal{  and  that, }\Vert F\Vert_{H^1(M_T)}^2\lesssim e^{2T\lambda}\lambda^2\sum_{\lambda_j\leq \lambda}|a_j|^2,
\end{equation}we deduce that
\begin{equation}
    C_{\alpha}\sum_{\lambda_j\leq \lambda}|a_j|^2\lesssim_{\alpha,T}\left( e^{2T\lambda}\lambda^2\sum_{\lambda_j\leq \lambda}|a_j|^2\right)^{\kappa}\left\Vert \sum_{\lambda_j\leq \lambda}a_j\rho_j(x)  \right\Vert_{L^2(\omega)}^{2(1-\kappa)},
\end{equation}from which we deduce that
\begin{align*}
    \left( \sum_{\lambda_j\leq \lambda}|a_j|^2\right)^{1-\kappa}\lesssim e^{2T\lambda}\lambda^2 \left\Vert \sum_{\lambda_j\leq \lambda}a_j\rho_j(x)  \right\Vert_{L^2(\omega)}^{2(1-\kappa)},
\end{align*}and in consequence
\begin{align*}
    \left( \sum_{\lambda_j\leq \lambda}|a_j|^2\right)^{\frac{1}{2}}\lesssim e^{{T\lambda/(1-\kappa)}}\lambda^{{1/(1-\kappa)} }\left\Vert \sum_{\lambda_j\leq \lambda}a_j\rho_j(x)  \right\Vert_{L^2(\omega)},
\end{align*}proving \eqref{ObservabilityInequality}. Indeed, one can estimate $e^{{T\lambda/(1-\kappa)}}\lambda^{{1/(1-\kappa)} }\lesssim e^{{T\lambda/(1-\kappa)}}.$

\begin{proof}[Proof of  Proposition \ref{Lemma:LR;Ineq}] 
We will procced with the proof of \eqref{Interpolation:Inequality}.
Note that, by normalising $\varkappa$ on $L^2(M)$ we can assume without loss of generality  that $\Vert \varkappa\Vert_{L^2(M)}=1.$

\subsubsection{ An auxiliar interpolation inequality on $[0,T+\varepsilon)$.}

Let $\varepsilon\in (0,1)$ be a positive parameter whose conditions will be imposed later.  Firstly, by replacing in  Lemma \ref{CarlemanInequality} the open interval $I_T:=(0,T)$ by $I_{T+\varepsilon}:=(0,T+\varepsilon),$ and with  $ M_{T+\varepsilon}:=M\times (0,T+\varepsilon), $ we shall make use of the following interpolation inequality:\\

\fbox{%
\parbox{\textwidth}{
 { \it 
  For any $T>0$ and all $\alpha\in (0,T/2),$ there exists $\kappa\in (0,1)$ such that
\begin{equation}\label{Global:Carleman:EstimateProofSpectralIneq} 
    \Vert \phi \Vert_{H^1(M\times (\alpha,T-\alpha))}\leq C\Vert \phi\Vert_{H^1(M_{T+\varepsilon})}^\kappa\left(\Vert (-\partial_t^2-\Delta_g)\phi \Vert_{L^2(M_{T+\varepsilon})} +\Vert \partial_t\phi(x,0)\Vert_{L^2(\omega)} \right)^{1-\kappa}
\end{equation} for all $\phi\in H^2(M_{T+\varepsilon})$ such that $\phi=0$ in $M\times \{0\}.$ }}
}\\

\subsubsection{ Construction of a suitable function $\phi$:}
Let us apply \eqref{Global:Carleman:EstimateProofSpectralIneq} with $\phi$ defined as follows. Consider $\psi\in C^\infty( M\times [0,T+\varepsilon])$ satisfying that
\begin{eqnarray}\psi(t):=
\begin{cases}C ,& \text{ }t\in [0,T],\,x\in M,
\\
0,& \text{ } t\in [T+\frac{3\varepsilon}{4},T+\varepsilon],\,x\in M,
\end{cases}
\end{eqnarray} where $0<C\leq \varepsilon. $ Assume that there exists $M_0>0,$ independent of $\varepsilon$ such that
\begin{equation}\label{BoundedNorm;phi:i}
    \Vert \psi^{(i)}\Vert_{L^\infty}\leq M_0,
\end{equation}for all $i\in \{1,2,3,4\}.$ We construct a function with these properties in Lemma \ref{Lemma:fucntion:psi} of Appendix II in Section \ref{Section:construction:psi}. Then, 
by considering the function
\begin{equation}
    F(x,t):=\sum_{\lambda_j\leq \lambda}\frac{\sinh(\lambda_jt)}{\lambda_j} a_j\rho_j(x),\,\,(x,t)\in M_T:=M\times [0,T],
\end{equation} and its extension to the whole interval $I_{T+\varepsilon}=[0,T+\varepsilon]$ by the constant function  equal to $F(x,T)$, that is
\begin{eqnarray}F(x,t)=
\begin{cases}\sum_{\lambda_j\leq \lambda}\frac{\sinh(\lambda_jt)}{\lambda_j} a_j\rho_j(x),\,\,(x,t) ,& \text{ }t\in [0,T],\,x\in M,
\\
F(x,T),& \text{ } t\in [T,T+\varepsilon],\,x\in M,
\end{cases}
\end{eqnarray}we consider 
\begin{equation}\label{Auxialiar:phi:}
    \phi(x,t):=F(x,t)\psi(t),\,\,(x,t)\in M_{T+\varepsilon}.
\end{equation}Note that $\phi(x,0)=F(x,0)=0=\phi(x,T+\varepsilon)=0.$ Note also that $\phi$ is an extension  of $F$ from $M_T$ to $M_{T+\varepsilon}.$ 

Now, let us consider the odd extension of $\phi$ to the whole interval $[-(T+\varepsilon),T+\varepsilon],$ that is,
\begin{equation}
    \phi(x,t)=-\phi(x,-t)=-\psi(-t)F(x,-t),\, -(T+\varepsilon)\,\leq t\leq 0.
\end{equation}

\subsubsection{Analysis of the term $\Vert \partial_t\phi(x,0)\Vert_{L^2(\omega)}$ } Note that $\psi$ has been defined on $[0,T+\varepsilon]$ and it has been extended to $[-(T+\varepsilon),0]$ using its odd extension, that is, the one defined by $\psi(-t)=-\psi(t),$ for $t\in [0,T+\varepsilon].$

An elementary calculation gives for any $t$ in a neigboorhood of $t=0,$ that
$$  \partial_t\phi(x,t)=\psi'(t)F(x,t)+\psi(t)\partial_tF(x,t). $$At $t=0$ we have
\begin{equation}
    \partial_t\phi(x,0)=\psi'(0)F(x,0)+\psi(0)\partial_tF(x,0)=\psi(0)\partial_tF(x,0),
\end{equation} and then
\begin{equation}
    \Vert \partial_t\phi(x,0)\Vert_{L^2(\omega)}=|\psi(0)|\times \Vert\partial_t F(x,0) \Vert_{L^2(\omega)}=|\psi(0)|\times\Vert\varkappa \Vert_{L^2(\omega)}.
\end{equation}

\subsubsection{Embedding of $M_T$ in a closed manifold $M\times \mathbb{T}_{T,\varepsilon},$  $\mathbb{T}_{T,\varepsilon}\cong \mathbb{S}^1$:} Now, we will proceed with a topological construction. 

It is clear that in the variable $t\in [-(T+\varepsilon),T+\varepsilon]$ the function $\phi$ can be extended periodically way to the whole line $\mathbb{R},$ or in other words, we can identify $\phi$ with a distribution (which we also denote by $\phi$) on the dilated torus
\begin{equation}
    \mathbb{T}_{T,\varepsilon}=\mathbb{R}/(2(T+\varepsilon)\mathbb{Z})=[-(T+\varepsilon),T+\varepsilon),
\end{equation}where in the resulting manifold $[-(T+\varepsilon),T+\varepsilon)$ we identify the end points  $-(T+\varepsilon)\sim T+\varepsilon.$ This construction allows the manifold $\mathbb{T}_{T,\varepsilon}$ to be diffeomorphic to the circle $\mathbb{S}^1,$ and in consequence the function $\phi\in \mathscr{D}'(M\times \mathbb{T}_{T,\varepsilon})$ is  smooth on the product space  $M\times \mathbb{T}_{T,\varepsilon}$ which is a compact manifold of $C^\infty$-class without boundary. In particular, we have the inclusion:
$\forall T,\varepsilon>0,\, M_T\subset M\times \mathbb{T}_{T,\varepsilon}. $
 
\subsubsection{Proof of the interpolation inequality} Now, we are ready for the proof of the interpolation inequality
$$ \Vert F \Vert_{H^1(M\times (\alpha,T-\alpha))}\leq C_{s_0,s_{00}}\Vert F\Vert_{H^1(M_{T})}^\kappa \Vert \varkappa \Vert_{L^2(\omega)}^{1-\kappa}.$$
In the identity \eqref{symmetry:l2:equation:Lapla} below, we will prove that with $\phi$ defined in \eqref{Auxialiar:phi:}, we have that
\begin{align*}
    \Vert (-\partial_t^2-\Delta_g)\phi(x,t) \Vert_{L^2(M_{T+\varepsilon})}= 1/\sqrt{2}  \Vert (-\partial_t^2-\Delta_g)\phi(x,t) \Vert_{L^2(M\times \mathbb{T}_{T,\varepsilon})}.
\end{align*}
The positivity condition
$$    E(x,D)\geq cI,\,\,\,c>0,$$ 
and the spectral mapping theorem gives the lower bound
$$    E(x,D)^{\frac{2}{\nu}}\geq c^{\frac{2}{\nu}}I,\,\,\,c>0,$$ 
which
gives the invertibility of the pseudo-differential operator $$ E(t,x,D,\partial_t)=-\partial_t^2+E(x,D)^{\frac{2}{\nu}}:H^2(M\times \mathbb{T}_{T,\varepsilon})\rightarrow L^2(M\times \mathbb{T}_{T,\varepsilon}),$$ that is, the operator 
,$$ E(t,x,D,\partial_t)^{-1}:L^2(M\times \mathbb{T}_{T,\varepsilon})\rightarrow H^2(M\times \mathbb{T}_{T,\varepsilon})$$  is bounded, see Lemma \ref{The:constant:B}.
Moreover,
\begin{align*}
      \Vert (-\partial_t^2-\Delta_g)\phi(x,t) \Vert_{L^2(M\times \mathbb{T}_{T,\varepsilon})}\leq C \Vert(-\partial_t^2+E(x,D)^{\frac{2}{\nu}})\phi(x,t)\Vert_{L^2(M\times \mathbb{T}_{T,\varepsilon})},
\end{align*}where the constant $C$ is independent of $\varepsilon>0.$ Indeed, 
$$ 
     \Vert (-\partial_t^2-\Delta_g)\phi(x,t) \Vert_{L^2(M\times \mathbb{T}_{T,\varepsilon})}$$
     $$= \Vert (-\partial_t^2-\Delta_g)(-\partial_t^2+E(x,D)^{2/\nu})^{-1} (-\partial_t^2+E(x,D)^{2/\nu})\phi(x,t) \Vert_{L^2(M\times \mathbb{T}_{T,\varepsilon})}.
$$ From Lemma \ref{Finite:Constant:CV}, the operator $ (-\partial_t^2-\Delta_g) (-\partial_t^2+E(x,D)^{2/\nu})^{-1} $ belongs to the class $\Psi^0_{1,\delta}(M)$ and the Calder\'on-Vaillancourt theorem gives its boundedness on $L^2,$ with its operator norm  bounded by a constant $C>0,$ independent of $\varepsilon>0.$
Consequently, 
$$ \Vert (-\partial_t^2-\Delta_g)(-\partial_t^2+E(x,D)^{2/\nu})^{-1} (-\partial_t^2+E(x,D)^{2/\nu})\phi(x,t) \Vert_{L^2(M\times \mathbb{T}_{T,\varepsilon})}$$ 
$$\leq C \Vert  (-\partial_t^2+E(x,D)^{2/\nu})\phi(x,t) \Vert_{L^2(M\times \mathbb{T}_{T,\varepsilon})}.
$$
In what follows we estimate this norm.
Let us denote 
\begin{equation*}
    Z_1:=   \Vert(-\partial_t^2+E(x,D)^{\frac{2}{\nu}})\phi(x,t)\Vert_{L^2(M\times \mathbb{T}_{T,\varepsilon})},
\end{equation*}
and let us keep in mind that the partial analysis above gives us the inequality
\begin{equation}\label{Carleman:partial:phi}
  \Vert \phi \Vert_{H^1(M\times (\alpha,T-\alpha))}\lesssim \Vert \phi\Vert_{H^1(M_{T+\varepsilon})}^\kappa(Z_1+|\psi(0)|\Vert \varkappa\Vert_{L^2(\omega)})^{1-\kappa} ,  
\end{equation}
where we have used that $$\psi(0)\varkappa=\partial_t\phi(x,0)=\psi(0)\partial_tF(x,0)=\psi(0)\sum_{\lambda_j\leq \lambda} a_j\rho_j(x).$$ Indeed, we record that
 for $t\in [0,T],$ we have
the identities
$$  \partial_t F(x,t)=\sum_{\lambda_j\leq \lambda}{\sinh(\lambda_jt)} a_j\rho_j(x),$$
$$  \partial_t^2 F(x,t)=\sum_{\lambda_j\leq \lambda}{\sinh(\lambda_jt)}\lambda_j a_j\rho_j(x).$$

\subsubsection{Estimate of $Z_1$} Observe that
$$ E(x,D)^{\frac{2}{\nu}}F(x,t)= \sum_{\lambda_j\leq \lambda}\frac{\sinh(\lambda_jt)}{\lambda_j} a_j E(x,D)^{\frac{2}{\nu}}(\rho_j)(x)=\sum_{\lambda_j\leq \lambda}\frac{\sinh(\lambda_jt)}{\lambda_j} a_j \lambda_j^2\rho_j(x) $$
$$ =\partial_t^2 F(x,t),\,\,t\in [0,T].  $$

Since $\phi(x,t)=\psi(t)F(x,t)$ on $[0, T],$ and $\psi$ is constant on $[0,T]$ we have that 
\begin{equation}\label{Cancellation}
    \forall x\in M,\,\forall t\in (0,T),\, (-\partial_t^2+E(x,D)^{\frac{2}{\nu}})\phi(x,t)=0.
\end{equation}
First, note that the following symmetry property is valid due to the identity $\phi(x,t)=-\phi(x,-t),$

\begin{equation}\label{symmetry:l2:equation}
   \Vert (-\partial_t^2+E(x,D)^{\frac{2}{\nu}})\phi(x,t) \Vert_{L^2(M\times \mathbb{T}_{T,\varepsilon})}^2= 2\Vert (-\partial_t^2+E(x,D)^{\frac{2}{\nu}})\phi(x,t) \Vert^2_{L^2(M_{T+\varepsilon})}.
\end{equation}
Moreover, once proved \eqref{symmetry:l2:equation}, we have in the particular case (where $E(x,D)=-\Delta_g$) of the Laplacian, the following inequality.
\begin{equation}\label{symmetry:l2:equation:Lapla}
    \Vert (-\partial_t^2-\Delta_g)\phi(x,t) \Vert_{L^2(M\times \mathbb{T}_{T,\varepsilon})}^2= 2\Vert (-\partial_t^2-\Delta_g)\phi(x,t) \Vert^2_{L^2(M_{T+\varepsilon})}.
\end{equation}
Indeed, for the proof of \eqref{symmetry:l2:equation} observe that
\begin{align*}
    & \Vert (-\partial_t^2+E(x,D)^{\frac{2}{\nu}})\phi(x,t) \Vert_{L^2(M\times \mathbb{T}_{T,\varepsilon})}^2\\
     &=\smallint\limits_{M}\smallint\limits_{-T-\varepsilon}^0\vert (-\partial_t^2+E(x,D)^{\frac{2}{\nu}})\phi(x,t) \vert^2 dt\,dx+\smallint\limits_{M}\smallint\limits_{0}^{T+\varepsilon}\vert (-\partial_t^2+E(x,D)^{\frac{2}{\nu}})\phi(x,t) \vert^2 dt\,dx\\
     &=\smallint\limits_{M}\smallint\limits_{-T-\varepsilon}^0\vert (-\partial_t^2+E(x,D)^{\frac{2}{\nu}})(-\phi(x,-t) \vert^2 dt\,dx+\smallint\limits_{M}\smallint\limits_{0}^{T+\varepsilon}\vert (-\partial_t^2+E(x,D)^{\frac{2}{\nu}})\phi(x,t) \vert^2 dt\,dx\\
     &=\smallint\limits_{M}\smallint\limits_{-T-\varepsilon}^0\vert \phi_{tt}(x,-t)-E(x,D)^{\frac{2}{\nu}}\phi(x,-t) \vert^2 dt\,dx+\smallint\limits_{M}\smallint\limits_{0}^{T+\varepsilon}\vert (-\partial_t^2+E(x,D)^{\frac{2}{\nu}})\phi(x,t) \vert^2 dt\,dx\\
     &=\smallint\limits_{M}\smallint\limits_{0}^{T+\varepsilon}\vert -\phi_{tt}(x,t)+E(x,D)^{\frac{2}{\nu}}\phi(x,t) \vert^2 dt\,dx+\smallint\limits_{M}\smallint\limits_{0}^{T+\varepsilon}\vert (-\partial_t^2+E(x,D)^{\frac{2}{\nu}})\phi(x,t) \vert^2 dt\,dx\\
&=2\smallint\limits_{M}\smallint\limits_{0}^{T+\varepsilon}\vert (-\partial_t^2+E(x,D)^{\frac{2}{\nu}})\phi(x,t) \vert^2 dt\,dx.
\end{align*}

The cancellation property in \eqref{Cancellation}, the symmetry property  $\phi(x,t)=-\phi(x,-t),$   and the positivity of the operator $(-\partial_t^2+E(x,D)^{\frac{2}{\nu}})$ on $L^2(M\times \mathbb{T}_{T,\varepsilon})$ (that is, making integration by parts with respect to the operator $(-\partial_t^2+E(x,D)^{\frac{2}{\nu}})$)  imply that
\begin{align*}
   \Vert (-\partial_t^2+E(x,D)^{\frac{2}{\nu}}) &\phi(x,t) \Vert_{L^2(M_{T+\varepsilon})}^2=\frac{1}{2}\Vert (-\partial_t^2+E(x,D)^{\frac{2}{\nu}})\phi(x,t) \Vert_{L^2(M\times \mathbb{T}_{T,\varepsilon})}^2\\
   &=\frac{1}{2}|((-\partial_t^2+E(x,D)^{\frac{2}{\nu}})\phi),(-\partial_t^2+E(x,D)^{\frac{2}{\nu}})\phi)_{L^2(M\times \mathbb{T}_{T,\varepsilon})}|\\
   &=\frac{1}{2}|((-\partial_t^2+E(x,D)^{\frac{2}{\nu}})^2\phi),\phi)_{L^2(M\times \mathbb{T}_{T,\varepsilon})}|\\
   &\leq \smallint\limits_{M}\smallint\limits_{ [0,T+\varepsilon)}|(-\partial_t^2+E(x,D)^{\frac{2}{\nu}})^2\phi(x,t)||\overline{\phi(x,t)}|dxdt\\
   &=\smallint\limits_{M}\smallint\limits_{ [T,T+\varepsilon)}|(-\partial_t^2+E(x,D)^{\frac{2}{\nu}})^2\phi(x,t)||{\phi(x,t)}|dxdt.
\end{align*}Therefore, we have the estimate
\begin{align*}
   \Vert (-\partial_t^2+E(x,D)^{\frac{2}{\nu}})\phi(x,t) \Vert_{L^2(M_{T+\varepsilon})}^2
   &\leq \smallint\limits_{M}\smallint\limits_{ [T,T+\varepsilon)}|(-\partial_t^2+E(x,D)^{\frac{2}{\nu}})^2\phi(x,t)||\phi(x,t)|dxdt\\
   &\leq \smallint\limits_{M}\smallint\limits_{ [T,T+\varepsilon)}|\phi(x,t)|dxdt\times \|(-\partial_t^2+E(x,D)^{\frac{2}{\nu}})^2\phi\|_{L^\infty}\\
   &=I\times II,
\end{align*}where 
\begin{equation*}
    I=\smallint\limits_{M}\smallint\limits_{ [T,T+\varepsilon)}|\phi(x,t)|dx\,dt,\,II= \|(-\partial_t^2+E(x,D)^{\frac{2}{\nu}})^2\phi\|_{L^\infty(M\times [T,T+\varepsilon))}.
\end{equation*} Now, we will estimate each one of these norms.\\

\subsubsection{Estimate for $I$} Note that
\begin{align*}
    I\leq \textnormal{Vol}(M)\times \varepsilon\times \Vert \phi\Vert_{L^\infty(M\times [T,T+\varepsilon])}=\textnormal{Vol}(M)\times \varepsilon \Vert \psi(t)\Vert_{L^\infty[T,T+\varepsilon]} \Vert F(x,t)\Vert_{L^\infty}.
\end{align*}Now,  for any $t$ fixed, and $s_0\in \mathbb{N}$ observe that
\begin{equation}
    |F(x,t)|\leq \sup_{0\leq s\leq s_0} \|(1+E(x,D))^{\frac{s}{\nu}}F(\cdot,t)\|_{L^\infty(M)}.
\end{equation}Using the Sobolev embedding theorem, we have that for any $s_{00}>n/2,$
\begin{align*}
  \sup_{0\leq s\leq s_0} \|(1+E(x,D))^{\frac{s}{\nu}}F(\cdot,t)\|_{L^\infty(M)}&\leq \sup_{0\leq s\leq s_0} \|(1+E(x,D))^{\frac{s+s_{00}}{\nu}}F(\cdot,t)\|_{L^2(M)}\\
  &\leq \sup_{0\leq s\leq s_0+s_{00}} \|(1+E(x,D))^{\frac{s}{\nu}}F(\cdot,t)\|_{L^2(M)}.
\end{align*}Now, let us use the spectral theory of the operator $(1+E(x,D))^{\frac{s}{\nu}}$ on $L^2(M).$
Since
\begin{align*}
   \|(1+E(x,D))^{\frac{s}{\nu}}F(\cdot,t)\|_{L^2(M)}^2 &=\left\Vert  \sum_{\lambda_j\leq \lambda}\frac{\sinh(\lambda_jt)}{\lambda_j}(1+\lambda_j^\nu)^{\frac{s}{\nu}}a_j\rho_j(x) \right\Vert^2_{L^2(M)} \\
   &=\sum_{\lambda_j\leq \lambda}\left|\frac{\sinh(\lambda_jt)}{\lambda_j}\right|^2(1+\lambda_j^\nu)^{\frac{2s}{\nu}}|a_j|^2\\
   &\lesssim \sum_{\lambda_j\leq \lambda}\left|\frac{\sinh(\lambda_jt)}{\lambda_j}\right|^2\lambda_j^{2s}|a_j|^2\lesssim \sum_{\lambda_j\leq \lambda}e^{\lambda_jt}\lambda_j^{2s-1}|a_j|^2\\
   &\lesssim_{s_0,s_{00}} e^{\lambda T} \sum_{\lambda_j\leq \lambda}|a_j|^2\\
   &= e^{\lambda T} \Vert \partial_t F(\cdot,0)\Vert_{L^2(M)}^2.
\end{align*}
In consequence, we deduce the inequality
\begin{equation}\label{auxialiar:s00}
    \forall s\in [0,s_0],\,\forall s_{00}>n/2,\,\,\Vert (1+E(x,D))^{\frac{s}{\nu}} F\Vert_{L^\infty}\lesssim_{s_0,s_{00}}e^{T\lambda/2}\|\partial_t F(\cdot,0)\Vert_{L^2(M)},
\end{equation}
as well as the Sobolev estimate
\begin{equation}\label{auxialiar:s0}
  \forall s_{00}>n/2,\, \forall s\in [0,s_0+s_{00}],\,\,\,\Vert (1+E(x,D))^{\frac{s}{\nu}} F\Vert_{L^2}\lesssim_{s_0,s_{00}}e^{T\lambda/2}\|\partial_t F(\cdot,0)\Vert_{L^2(M)}.
\end{equation}
With $s_0=0,$ we have that $\Vert  F\Vert_{L^\infty}\lesssim_{s_0,s_{00}}e^{ T\lambda /2}\|\partial_t F(\cdot,0)\Vert_{L^2(M)}.$ The analysis above gives us the inequality:
\begin{equation}\label{Estimate:I}
    I\lesssim \textnormal{Vol}(M)\times \varepsilon \Vert \psi(t)\Vert_{L^\infty[T,T+\varepsilon]} e^{T\lambda/2} \|\partial_t F(\cdot,0)\Vert_{L^2(M)}=\textnormal{Vol}(M)\times \varepsilon \Vert \psi(t)\Vert_{L^\infty[T,T+\varepsilon]} e^{T\lambda/2} ,
\end{equation} where we have used that $\|\partial_t F(\cdot,0)\Vert_{L^2(M)}=\Vert \varkappa\Vert_{L^2(M)}=1.$ Summarising 
 $$   I\lesssim \textnormal{Vol}(M)\times \varepsilon \Vert \psi(t)\Vert_{L^\infty[T,T+\varepsilon]} e^{T\lambda/2}. $$

\subsubsection{Estimate for $II$:} To estimate the second term, we start by observing the inequality
\begin{equation*}
    II= \|(-\partial_t^2+E(x,D)^{\frac{2}{\nu}})^2\phi\|_{L^\infty(M\times [T,T+\varepsilon))}=\|(-\partial_t^2+E(x,D)^{\frac{2}{\nu}})^2[\psi(t)F(x,T)]\|_{L^\infty(M\times [T,T+\varepsilon))}.
\end{equation*}Since
\begin{align*}
   (-\partial_t^2+ E(x,D)^{\frac{2}{\nu}})^2 &[\psi(t)F(x,T)]=(-\partial_t^2+ E(x,D)^{\frac{2}{\nu}})(-\partial_t^2+ E(x,D)^{\frac{2}{\nu}})[\psi(t)F(x,T)]\\
   &=(-\partial_t^2+E(x,D)^{\frac{2}{\nu}})[-\psi_{tt}(t)F(x,T)+\psi(t)E(x,D)^{\frac{2}{\nu}}F(x,T)]\\
   &=\psi^{(4)}(t)F(x,T)-2\psi_{tt}(t)E(x,D)^\frac{2}{\nu}F(x,T)\\
   &\,\,\,\,+\psi(t)E(x,D)^{\frac{4}{\nu}}(F(x,T)),
\end{align*}for $s_0\geq 4,$ and with $s_{00}>n/2,$  the Sobolev inequality in \eqref{auxialiar:s00} implies that
\begin{align*}
  &  \|(-\partial_t^2+E(x,D)^{\frac{2}{\nu}})^2[\psi(t)F(x,T)]\|_{L^\infty(M\times [T,T+\varepsilon))}\\
  &\leq \Vert \psi^{(4)}\Vert_{L^\infty[T,T+\varepsilon]}\Vert F(x,T)\Vert_{L^\infty}\\
  &+2\Vert\psi_{tt} \Vert_{L^\infty[T,T+\varepsilon]}\Vert E(x,D)^{\frac{2}{\nu}}F(x,T) \Vert_{L^\infty}+\Vert\psi\Vert_{{L^\infty[T,T+\varepsilon]}}\Vert E(x,D)^{\frac{4}{\nu}}F(x,T) \Vert_{L^\infty}\\
  &\leq \Vert \psi^{(4)}\Vert_{L^\infty[T,T+\varepsilon]}\Vert F(x,T)\Vert_{L^\infty}\\
  &+2\Vert\psi_{tt} \Vert_{L^\infty[T,T+\varepsilon]}\Vert(1+ E(x,D))^{\frac{s_{00}+2}{\nu}}F(x,T) \Vert_{L^2}\\
  &+\Vert\psi\Vert_{{L^\infty[T,T+\varepsilon]}}\Vert (1+ E(x,D))^{\frac{s_{00}+4}{\nu}}F(x,T) \Vert_{L^2}\\
  & \lesssim_{s_0,s_{00}}e^{{T} {\lambda}/2 }\|\partial_t F(\cdot,0)\Vert_{L^2(M)} (\Vert \psi^{(4)}\Vert_{L^\infty[T,T+\varepsilon]}+2\Vert\psi_{tt} \Vert_{L^\infty[T,T+\varepsilon]}+\Vert\psi\Vert_{{L^\infty[T,T+\varepsilon]}})\\
  &=e^{{T} {\lambda}/2 }\|\varkappa\Vert_{L^2(M)} (\Vert \psi^{(4)}\Vert_{L^\infty[T,T+\varepsilon]}+2\Vert\psi_{tt} \Vert_{L^\infty[T,T+\varepsilon]}+\Vert\psi\Vert_{{L^\infty[T,T+\varepsilon]}})\\
 &=e^{{T} {\lambda}/2 } (\Vert \psi^{(4)}\Vert_{L^\infty[T,T+\varepsilon]}+2\Vert\psi_{tt} \Vert_{L^\infty[T,T+\varepsilon]}+\Vert\psi\Vert_{{L^\infty[T,T+\varepsilon]}}) .
\end{align*}
\subsubsection{Estimate for $Z_1$}
Having proved these estimates for $I$ and  $II,$ we have that
\begin{align*}
  &  \Vert (-\partial_t^2+E(x,D)^{\frac{2}{\nu}})\phi(x,t) \Vert_{L^2(M_{T+\varepsilon})}^2\\
    &\lesssim \textnormal{Vol}(M)\times \varepsilon \Vert \psi(t)\Vert_{L^\infty[T,T+\varepsilon]} e^{T\lambda/2}\times e^{{T} {\lambda}/2 } (\Vert \psi^{(4)}\Vert_{L^\infty[T,T+\varepsilon]}+2\Vert\psi_{tt} \Vert_{L^\infty[T,T+\varepsilon]}+\Vert\psi\Vert_{{L^\infty[T,T+\varepsilon]}})\\
    &=\textnormal{Vol}(M)\times \varepsilon \Vert \psi(t)\Vert_{L^\infty[T,T+\varepsilon]} e^{T\lambda} (\Vert \psi^{(4)}\Vert_{L^\infty[T,T+\varepsilon]}+2\Vert\psi_{tt} \Vert_{L^\infty[T,T+\varepsilon]}+\Vert\psi\Vert_{{L^\infty[T,T+\varepsilon]}}).
\end{align*}
\subsubsection{Final Analysis}
The estimates above for $Z_1$  lead to the following inequality in view of the interpolation inequality
\begin{align*}
 &\Vert \phi\Vert_{H^1(M\times (\alpha,T-\alpha))}\\
 &\lesssim \Vert \phi\Vert_{H^1(M_{T+\varepsilon})}^\kappa((\textnormal{Vol}(M)\times \varepsilon \Vert \psi(t)\Vert_{L^\infty[T,T+\varepsilon]} e^{T\lambda} (\Vert \psi^{(4)}\Vert_{L^\infty[T,T+\varepsilon]}+2\Vert\psi_{tt} \Vert_{L^\infty[T,T+\varepsilon]}+\Vert\psi\Vert_{{L^\infty[T,T+\varepsilon]}}))^{\frac{1}{2}}\\
 &+|\psi(0)|\Vert \varkappa \Vert_{L^2(\omega)})^{1-\kappa} .
 \end{align*} 
Now,  dividing both sides of this inequality by $|\phi(0)|$ and using that $\psi(0)=\psi(T)=\psi(t),$ $0\leq t\leq T,$ we get 
\begin{align*}
 & \Vert F(x,t)\Vert_{H^1(M\times (\alpha,T-\alpha))}=\left\Vert \frac{\psi(t) F(x,t)}{\psi(0)}\right\Vert_{H^1(M\times (\alpha,T-\alpha))}\\
 &\lesssim \left\Vert \frac{\psi(t)F(x,t)}{\psi(0)}\right\Vert_{H^1(M_{T+\varepsilon})}^\kappa\\
 &\times\frac{1}{|\psi(0)|^{1-\kappa}} (  (\textnormal{Vol}(M)\times \varepsilon \Vert \psi(t)\Vert_{L^\infty[T,T+\varepsilon]} e^{T\lambda} (\Vert \psi^{(4)}\Vert_{L^\infty[T,T+\varepsilon]}+2\Vert\psi_{tt} \Vert_{L^\infty[T,T+\varepsilon]}+\Vert\psi\Vert_{{L^\infty[T,T+\varepsilon]}}))^{\frac{1}{2}}\\
 &+|\psi(0)|\Vert \varkappa \Vert_{L^2(\omega)})^{1-\kappa} \\
 &=\left\Vert \frac{\psi(t)F(x,t)}{\psi(0)}\right\Vert_{H^1(M_{T+\varepsilon})}^\kappa\\
 &\times (\left(\textnormal{Vol}(M)\times \varepsilon  e^{T\lambda}\frac{\Vert \psi(t)\Vert_{L^\infty[T,T+\varepsilon]}}{|\psi(0)|} \left(\frac{\Vert \psi^{(4)}\Vert_{L^\infty[T,T+\varepsilon]}}{\| \psi(0)|}+\frac{2\Vert\psi_{tt} \Vert_{L^\infty[T,T+\varepsilon]}}{| \psi(0)|}+1\right)\right)^{\frac{1}{2}}\\
 &+\Vert \varkappa \Vert_{L^2(\omega)})^{1-\kappa}\\
 &=\left\Vert \frac{\psi(t)F(x,t)}{\psi(T)}\right\Vert_{H^1(M_{T+\varepsilon})}^\kappa\\
 &\times (\left(\textnormal{Vol}(M)\times \varepsilon  e^{T\lambda}\frac{\Vert \psi(t)\Vert_{L^\infty[T,T+\varepsilon]}}{|\psi(T)|} \left(\frac{\Vert \psi^{(4)}\Vert_{L^\infty[T,T+\varepsilon]}}{|\psi(T)|}+\frac{2\Vert\psi_{tt} \Vert_{L^\infty[T,T+\varepsilon]}}{|\psi(T)|}+1\right)\right)^{\frac{1}{2}}\\
 &+\Vert \varkappa \Vert_{L^2(\omega)})^{1-\kappa}. 
 \end{align*} Consequently, the previous analysis leads to the auxiliary inequality
 \begin{align*}
    &  \Vert F(x,t)\Vert_{H^1(M\times (\alpha,T-\alpha))}\lesssim
   \left\Vert \frac{\psi(t)F(x,t)}{\psi(T)}\right\Vert_{H^1(M_{T+\varepsilon})}^\kappa\\
 &\times (\left(\textnormal{Vol}(M)\times \varepsilon  e^{T\lambda}\frac{\Vert \psi(t)\Vert_{L^\infty[T,T+\varepsilon]}}{|\psi(T)|} \left(\frac{\Vert \psi^{(4)}\Vert_{L^\infty[T,T+\varepsilon]}}{|\psi(T)|}+\frac{2\Vert\psi_{tt} \Vert_{L^\infty[T,T+\varepsilon]}}{|\psi(T)|}+1\right)\right)^{\frac{1}{2}}\\
 &+\Vert \varkappa \Vert_{L^2(\omega)})^{1-\kappa},
 \end{align*}
 and taking the limit when $\varepsilon\rightarrow 0^+$ in both sides of this estimate
we  conclude the expected inequality,
\begin{equation}\label{Expected:Inequality}
    \Vert F \Vert_{H^1(M\times (\alpha,T-\alpha))}\leq C_{s_0,s_{00}}\Vert F\Vert_{H^1(M_{T})}^\kappa \Vert \varkappa \Vert_{L^2(\omega)}^{1-\kappa},
\end{equation}
where we have used  that $\psi(t)/\psi(0)=\psi(t)/\psi(T)=1,$ $0\leq t\leq T,$ the smoothness of $\psi,$ the fact that (see Lemma \ref{Lemma:fucntion:psi}) $$ \lim_{\varepsilon\rightarrow0^+}\Vert \psi^{(4)}\Vert_{L^\infty[T,T+\varepsilon]}=\psi_{tttt}(T)=\lim_{\varepsilon\rightarrow0^+}\Vert \psi_{tt}\Vert_{L^\infty[T,T+\varepsilon]}=\psi_{tt}(T)=0,$$   and the properties
\begin{itemize}
    \item 
    \begin{equation}
        \lim_{
    \varepsilon\rightarrow 0^+
    }\left\Vert{\psi(t)}/{\psi(T)} \right\Vert_{L^\infty[T,T+\varepsilon]}=1,
    \end{equation}and
    \item 
    \begin{equation}\label{2:33}
 \lim_{\varepsilon\rightarrow 0}\left\Vert{\psi(t)F(x,t)}/{\psi(0)}\right\Vert_{H^1(M_{T+\varepsilon})}= \Vert F(x,t)\Vert_{H^1(M_{T})}.  
    \end{equation}
     \end{itemize}
For the proof of \eqref{2:33} note that for $\tilde{\psi}:=\psi(t)/\psi(0),$ and using that $F(x,t)=F(x,T)$ if $0\leq t\leq T+\varepsilon,$ we have
$$ 
  \lim_{\varepsilon\rightarrow 0}\left\Vert{\psi(t)F(x,t)}/{\psi(0)}\right\Vert^2_{H^1(M_{T+\varepsilon})}=\lim_{\varepsilon\rightarrow 0}\sum_{j=0,1}\smallint\limits_{0}^{T+\varepsilon}\Vert\partial_t^{(j)}(\tilde{\psi}(t)F(x,t))\Vert^2_{H^1(M)}dt $$ 
  $$ 
  =\lim_{\varepsilon\rightarrow 0}\sum_{j=0,1}\smallint\limits_{0}^{T}\Vert\partial_t^{(j)}(\tilde{\psi}(t)F(x,t))\Vert^2_{H^1(M)}dt+\lim_{\varepsilon\rightarrow 0}\sum_{j=0,1}\smallint\limits_{T}^{T+\varepsilon}\Vert\partial_t^{(j)}(\tilde{\psi}(t)F(x,t))\Vert^2_{H^1(M)}dt $$ 
  $$ 
  =\sum_{j=0,1}\smallint\limits_{0}^{T}\Vert\partial_t^{(j)}(F(x,t))\Vert^2_{H^1(M)}dt+\lim_{\varepsilon\rightarrow 0}\sum_{j=0,1}\smallint\limits_{T}^{T+\varepsilon}\Vert\tilde{\psi}^{(j)}(t)F(x,T)\Vert^2_{H^1(M)}dt $$   $$= \Vert F(x,t)\Vert_{H^1(M_{T})}^2, $$
where we have used that when $t\rightarrow T, $ $\tilde{\psi}^{(j)}(t)\rightarrow 0.$
Now, from \eqref{Expected:Inequality} we can follow the standard Lebeau-Robbiano argument that has been described at the beginning of the section to conclude the proof of the spectral inequality.  Having proved \eqref{ObservabilityInequality}, the proof  of Proposition \ref{Lemma:LR;Ineq} is complete.
\end{proof}

\begin{proof}[Proof of Theorem \ref{The:Spectral:Inequality}] Let $\mu,c>0$ be two positive parameters and assume that $0<c<\mu^\nu$. Define
$$\tilde{E}(x,D)=E(x,D)+c.$$ Observe that the principal symbol  $\tilde{E}(x,\xi)$ of the operator $\tilde{E}(x,D)$  satisfies the lower bound $$\tilde{E}(x,\xi)\geq c,\,\,\xi\in T^*_xM,$$ and also  that $\tilde{E}(x,D)\geq cI.$  If  $\{\mu_j:=\lambda_j^\nu,\rho_j\}$ are the corresponding spectral data $$ E(x,D)\rho_j=\lambda_j^\nu \rho_j,\,\,\lambda_j\geq 0,$$ of $E(x,D),$ then  $\{\mu_j+c:=\lambda_j^\nu+c,\rho_j\}$ are the corresponding spectral data $$\tilde{ E}(x,D)\rho_j=(\lambda_j^\nu+c) \rho_j,\,\,\lambda_j\geq 0$$  of the operator $\tilde{E}(x,D).$  Let $\lambda:=(\mu^\nu+c)^{\frac{1}{\nu}}.$ From Proposition \ref{Lemma:LR;Ineq} we deduce the spectral inequality
\begin{equation}\label{ObservabilityInequality:Proof:II}
    \left(\sum_{ (\lambda_j^\nu+c)^{\frac{1}{\nu} }\leq \lambda}a_j^2\right)^\frac{1}{2}\leq C_1e^{C_2{\lambda}}\left\Vert \sum_{{ (\lambda_j^\nu+c)^{\frac{1}{\nu} }\leq \lambda}}a_j\rho_j(x)  \right\Vert_{L^2(\omega)}.
\end{equation}Note that $(\lambda_j^\nu+c)^{\frac{1}{\nu} }\leq \lambda$ becomes equivalent to the inequality $\lambda_j\leq \mu$ and since $0<c<\mu^\nu,$ then $\lambda<2\mu.$ Thus, we have proved the spectral inequality 
\begin{equation}
    \left(\sum_{\lambda_j\leq \mu}a_j^2\right)^\frac{1}{2}\leq C_1e^{ 2C_2{\mu}}\left\Vert \sum_{\lambda_j\leq \mu}a_j\rho_j(x)  \right\Vert_{L^2(\omega)}.
\end{equation}In consequence the proof of Theorem \ref{The:Spectral:Inequality} is complete.
\end{proof}

\subsection{Null-controllability for the diffusion model} Now, we present the proof of Theorem \ref{Main:theorem} which we present in the following way.
\begin{theorem}\label{Main:theorem:statement} Let $E(x,D)$ be a positive and elliptic pseudo-differential operator of order $\nu>0$ in the H\"ormander class $\Psi^\nu_{\rho,\delta}(M).$  Let $(x,\xi)\in T^*M,$ and assume that for any $\xi\neq 0,$ $E(x,\xi)>0$ is strictly positive.  

Then, for any $\alpha>1/\nu,$  the fractional diffusion model
\begin{equation}\label{Main:statement}
\begin{cases}u_t(x,t)+ E(x,D)^\alpha u(x,t)=g(x,t)\cdot 1_\omega (x) ,& (x,t)\in M\times (0,T),
\\u(0,x)=u_0,\end{cases}
\end{equation} is null-controllable at any time $T>0,$ that is, there exists an input function $g\in L^2(M)$ such that for any $x\in M,$ $u(x,T)=0.$ 
\end{theorem}
\begin{proof}
The spectral inequality in \eqref{ObservabilityInequality} allows us to make use of Theorem \ref{Miller:Theorem} with $A=E(x,D)^\alpha$. Indeed, since $A^\gamma= E(x,D)^{\alpha\gamma}=E(x,D)^\frac{1}{\nu},$ satisfies \eqref{ObservabilityInequality} (that is, the inequality \eqref{Espectral:Inequality:Hilbert} holds) for $\alpha\gamma=1/\nu.$ Because $\gamma\in (0,1)$ if an only if $\alpha>1/\nu,$ Theorem  \ref{Miller:Theorem} guarantees that this inequality on the fractional order $\alpha$ is a sufficient condition in order that \eqref{Main:statement} will be null-controllable in time $T>0.$ The proof of Theorem \ref{Main:theorem:statement} is complete.
\end{proof}
\subsection{The controllability cost for the diffusion model over short times}
 In the following result we analyse the controllability cost of the model \eqref{Main:statement} when the time is small.
 \begin{corollary}\label{Cost:Control} The controllability cost $C_T$ for the fractional heat equation \eqref{Main:statement} over short times $T\in (0,1)$ satisfies 
 \begin{equation}
     C_T\leq C_1e^{C_2T^{-\beta}},
 \end{equation}where $\beta>1/(\alpha\nu-1).$
 \end{corollary}
 \begin{proof}
 For the proof, note that $A^\gamma= E(x,D)^{\alpha\gamma}=E(x,D)^\frac{1}{\nu},$ satisfies \eqref{ObservabilityInequality}  for $\alpha\gamma=1/\nu.$ Then, from Theorem  \ref{Miller:Theorem} we have the estimate $  C_T\leq C_1e^{C_2T^{-\beta}}$ for any $\beta>\gamma/(\gamma-1)=1/(\alpha\nu-1).$ The proof of Corollary \ref{Cost:Control} is complete.
 \end{proof}

\section{Conclusions and some open problems}

This work deals with the null-controllability of the fractional heat equation associated to an invertible positive elliptic pseudo-differential operator on a compact Riemannian manifold. Concrete examples are indeed invertible positive elliptic partial differential operators of (necessarily) even order. 

In view of the difficulties when working with Carleman estimates for high-order differential operators or even for pseudo-differential operators which in general are non-local operators, we have shown that the theory of pseudo-differential operators can be effectively used. Moreover, this theory was strongly used in deriving the interpolation inequality in \eqref{Expected:Inequality}, that is  \begin{equation}
    \Vert F \Vert_{H^1(M\times (\alpha,T-\alpha))}\leq C_{s_0,s_{00}}\Vert F\Vert_{H^1(M_{T})}^\kappa \Vert \varkappa \Vert_{L^2(\omega)}^{1-\kappa}.
\end{equation} 
In conclusion with this work
\begin{itemize}
    \item we show that the theory of pseudo-differential operators is an effective microlocal tool for the control theory of parabolic problems, even when we have the lack of Carleman estimates.
\end{itemize}
The following are some open problems that arise from this investigation.
\begin{itemize}
    \item To extend the spectral inequality in \eqref{Expected:Inequality} to the case of suitable boundary value problem on compact  manifolds with smooth boundary. Just to mention a few, the case of the Laplacian $-\Delta_g$ with the corresponding Dirichlet boundary conditions was derived in the classical Lebeau-Robianno result in \cite{LabeauRobbiano1995}, and the case of the bi-Laplacian $(-\Delta_g)^2$ endowed with the so called {\it clamped} boundary conditions was analysed in the recent work of Rousseau and Robbiano \cite{RousseauRobbiano2020}. Even, in \cite{RousseauRobbiano2020} the authors leave open the question of the validity of the spectral inequality for the powers $(-\Delta_g)^k,$ $k\geq 3,$ when $\partial M\neq 0.$
    \item The following problem could be of interest in global analysis: to analyse for which kind of manifolds $(M,g)$ the condition $\rho\geq 1-\delta$ can be removed in Theorem \ref{The:Spectral:Inequality}. Indeed, the inequality $\rho\geq 1-\delta$ is the required condition in order to have the invariance of the pseudo-differential calculus under changes of coordinates. We propose the following conjecture: in the case where $M$ is a manifold with symmetries, and under a suitable positivity condition on the symbol of the operator $E(x,D)$, for any $0\leq \delta<\rho\leq 1,$ the spectral inequality in  Theorem \ref{The:Spectral:Inequality} is still valid.
\end{itemize}

\section{Appendix I: Pseudo-differential operators and null-controllability on Hilbert spaces}\label{Preliminaries}
In this section we present the preliminaries about the theory of pseudo-differential operators used in this work as well as some results on the null-controllability for the Cauchy problem on Hilbert spaces. We use the following notation.
\begin{itemize}
    \item For any pair $(X,Y)$ of Hilbert spaces, we denote by $\mathscr{B}(X,Y)$ the family of bounded and linear operators $T:X\rightarrow Y.$
   
    \item The spectrum of a densely defined linear operator $A:\textnormal{Dom}(A)\subset{X}\rightarrow X$ will be denoted by $\sigma(A).$ 
\end{itemize}

\subsection{H\"ormander classes of pseudo-differential operators on compact manifolds}\label{S2.1}
 In this subsection we present the construction of the classes of pseudo-differential operators on compact manifolds by using local coordinate systems (see H\"ormander \cite{Hormander1985III} and e.g. M. Taylor \cite{Taylorbook1981} for an introductory presentation).
 
Let us briefly  introduce these  classes starting with the definition in the Euclidean setting. 

\begin{definition}[Symbol classes]
Let $U$ be an open  subset of $\mathbb{R}^n.$ We  say that  the {\it symbol} $a\in C^{\infty}(U\times \mathbb{R}^n, \mathbb{C})$ belongs to the H\"ormander class of order $m$ and of $(\rho,\delta)$-type, $S^m_{\rho,\delta}(U\times \mathbb{R}^n),$ $0\leqslant \rho,\delta\leqslant 1,$ if for every compact subset $K\subset U$ and for all $\alpha,\beta\in \mathbb{N}_0^n$, the symbol inequalities
\begin{equation}\label{seminorms}
  |\partial_{x}^\beta\partial_{\xi}^\alpha a(x,\xi)|\leqslant C_{\alpha,\beta,K}(1+|\xi|)^{m-\rho|\alpha|+\delta|\beta|},
\end{equation} hold true uniformly in $x\in K$ for all  $\xi\in \mathbb{R}^n.$
\end{definition}

 Then, a continuous linear operator $A:C^\infty_0(U) \rightarrow C^\infty(U)$ 
is a pseudo-differential operator of order $m$ of  $(\rho,\delta)$-type, if there exists
a symbol $a\in S^m_{\rho,\delta}(U\times \mathbb{R}^n)$ such that
\begin{equation*}
    Af(x)=\smallint\limits_{\mathbb{R}^n}e^{2\pi i x\cdot \xi}a(x,\xi)(\mathscr{F}_{\mathbb{R}^n}{f})(\xi)d\xi,
\end{equation*} for all $f\in C^\infty_0(U),$ where
$$
    (\mathscr{F}_{\mathbb{R}^n}{f})(\xi):=\smallint\limits_Ue^{-i2\pi x\cdot \xi}f(x)dx
$$ is the  Euclidean Fourier transform of $f$ at $\xi\in \mathbb{R}^n.$ In this case we denote the class of pseudo-differential operators with symbols in the family $S^m_{\rho,\delta}(U\times \mathbb{R}^n)$ as $\Psi^m_{\rho,\delta}(U).$

Once the definition of H\"ormander classes on open subsets of $\mathbb{R}^n$ is given, it can be extended to smooth manifolds as follows. \begin{remark}[Pseudo-differential operators on compact manifolds]
Given a compact manifold without boundary $M,$ a linear continuous operator $A:C^\infty_0(M)\rightarrow C^\infty(M) $ is a pseudo-differential operator of order $m$ of $(\rho,\delta)$-type, with $ \rho\geqslant   1-\delta, $ and $0\leq \delta<\rho\leq 1,$  if for every local  coordinate patch $\omega: M_{\omega}\subset M\rightarrow U_{\omega}\subset \mathbb{R}^n,$
and for every $\phi,\psi\in C^\infty_0(U_\omega),$ the operator
\begin{equation*}
    Tu:=\psi(\omega^{-1})^*A\omega^{*}(\phi u),\,\,u\in C^\infty(U_\omega),
\end{equation*} is a standard pseudo-differential operator with symbol $a_T\in S^m_{\rho,\delta}(U_\omega\times \mathbb{R}^n).$ As usually, $\omega^{*}$ and $(\omega^{-1})^*$ are the pullbacks, induced by the maps $\omega$ and $\omega^{-1}$ respectively. In this case we write $A\in \Psi^m_{\rho,\delta}(M).$ 
\end{remark} 
\begin{remark}[The principal symbol]
The symbol of a pseudo-differential operator $A$ is unique as an element in $\Psi^m_{\rho,\delta}(M)/\Psi^m_{\rho,\delta}(M).$ To this element we will call the {\it principal symbol of} $A.$  We denote it by
\begin{equation}
    a_m(x,\xi),\,\, (x,\xi)\in T^{*}M.
\end{equation}
The main feature of the principal symbol of a pseudo-differential operator is that it remains invariant under coordinates changes.
\end{remark}

In the next theorem we describe some fundamental properties of the H\"ormander  calculus  \cite{Hormander1985III}.
\begin{theorem}\label{calculus} Let $0\leqslant \delta<\rho\leqslant 1,$ be such that $\rho\geq 1-\delta.$ Then $\Psi^\infty_{\rho,\delta}(M):=\cup_{m\in \mathbb{R}} \Psi^m_{\rho,\delta}(M)$ is an algebra of operators stable under compositions and adjoints, that is:
\begin{itemize}
    \item [-] the mapping $A\mapsto A^{*}:\Psi^{m}_{\rho,\delta}(M)\rightarrow \Psi^{m}_{\rho,\delta}(M)$ is a continuous linear mapping between Fr\'echet spaces. 
\item [-] The mapping $(A_1,A_2)\mapsto A_1\circ A_2: \Psi^{m_1}_{\rho,\delta}(M)\times \Psi^{m_2}_{\rho,\delta}(M)\rightarrow \Psi^{m_1+m_2}_{\rho,\delta}(M)$ is a continuous bilinear mapping between Fr\'echet spaces.
\end{itemize}Moreover, any operator in the class   $ \Psi^{0}_{\rho,\delta}(M)$ admits a  bounded extension from $L^2(G)$ to  $L^2(G).$
\end{theorem} 
\begin{remark}
    The $L^2$-boundedness result in Theorem \ref{calculus} is the microlocalised version of the Calder\'on-Vaillancourt theorem, see \cite{CalderonVaillancourt71}. Moreover, if $A\in \Psi^0_{\rho,\delta}(\mathbb{R}^n)$ is such that $0\leq \delta\leq \rho\leq 1,$  $\rho\neq 1,$ then
    \begin{equation}
        \Vert A\Vert_{\mathscr{B}(L^2)}\lesssim \sup_{|\alpha|+|\beta|\leq [\frac{n}{2}]+1} C_{\alpha,\beta},
    \end{equation}where
    $$ C_{\alpha,\beta}:=\sup_{(x,\xi)\in \mathbb{R}^{2n}}   (1+|\xi|)^{\rho|\alpha|-\delta|\beta|} |\partial_{x}^\beta\partial_{\xi}^\alpha a(x,\xi)|.$$
\end{remark}
\begin{remark}\label{remark:CVTh}
    If $U$ is an open subset of a close manifold $M$ of dimension $n,$ and $f\in C^\infty_0(U),$ then, by  microlocalising   $A\in \Psi^0_{\rho,\delta}(M)$ when $0\leq \delta < \rho\leq 1,$ and $\rho\geq 1-\delta,$ one has
    \begin{equation}
        \Vert A f\Vert_{L^2}\lesssim_{U} \sup_{|\alpha|+|\beta|\leq [\frac{n}{2}]+1} C_{\alpha,\beta,U}\Vert f\Vert_{L^2(M)},
    \end{equation}where
    $$ C_{\alpha,\beta,U}:=\sup_{(x,\xi)\in U\times\mathbb{R}^{n}}   (1+|\xi|)^{\rho|\alpha|-\delta|\beta|} |\partial_{x}^\beta\partial_{\xi}^\alpha a(x,\xi)|,$$ by making the identification of $U$ with an open subset of $\mathbb{R}^n.$
\end{remark}
In this work we will use the functional calculus for elliptic pseudo-differential operators on $M.$ We record that a pseudo-differential operator $A\in \Psi^m_{\rho,\delta}(M)$ is  elliptic if  in any local coordinate system $U$, there exists $R>0,$ such that its symbol satisfies uniformly on any compact subset $K\subset{U}$ the growth estimate
\begin{equation}\label{Ellipticity}
   C_1 (1+|\xi|)^m\leq   |a(x,\xi)|\leq C_2 (1+|\xi|)^m, \, |\xi|\geq R,
\end{equation}uniformly in $x\in K,$ and $\xi\in \mathbb{R}^n.$ The ellipticity condition \eqref{Ellipticity} is not enough for constructing the complex powers of an elliptic operator. We present such a construction in the following sub-section. However, one of the main aspects of the spectral theory of an elliptic operator is that it spectrum is purely discrete  \cite{Hormander1985III}. 
\subsection{Complex powers of an elliptic pseudo-differential operator} 

In this subsection  to a complex sector 
$$  \Lambda\subset\mathbb{ C}, $$
we will associate a class class of elliptic operators $\Psi^m_{\rho,\delta}(M;\Lambda)$.  In the 
applications this will, as a rule, be an angle with the vertex at point $z_0\in \mathbb{C}$.  For this we will follow Shubin \cite{Shubin1987}. For a closed manifold $M$ we require the conditions $\rho\geq 1-\delta,$ and $0\leq \delta<\rho\leq 1.$  We define these classes as follows. For an open subset $U\subset \mathbb{R}^n,$ let $a(x,\theta,\lambda)$ be a function on $U\times \mathbb{R}^n\times\Lambda.$ For $m\in \mathbb{R},$ and $d_\Lambda>,0$ we say that $a$ belongs to the class $$S^m_{\rho,\delta}(U\times \mathbb{R}^n; \Lambda)=S^m_{\rho,\delta,d_\Lambda}(U\times \mathbb{R}^n; \Lambda)$$ if the following conditions are satisfied,
\begin{itemize}
    \item $\forall\lambda_0\in \Lambda,\, a(x,\theta,\lambda_0)\in C^\infty(U\times \mathbb{R}^n).$
    \item For every compact subset $K\subset U$ and for all $\alpha,\beta\in \mathbb{N}_0^n$, the symbol inequalities
\begin{equation*}
  |\partial_{x}^\beta\partial_{\xi}^\alpha a(x,\theta,\lambda)|\leqslant C_{\alpha,\beta,K}(1+|\theta|+|\lambda|^{\frac{1}{d_\Lambda}})^{m-\rho|\alpha|+\delta|\beta|},
\end{equation*} hold true uniformly in $x\in K,$ for all  $\theta\in \mathbb{R}^n,$ and all $\lambda\in \Lambda.$ 
\end{itemize} By simplicity of the notation we have omitted the dependence of these classes with respect to the parameter $d_\Lambda.$

Then, to any $a\in S^m_{\rho,\delta}(U\times \mathbb{R}^n; \Lambda)$ we can associate a pseudo-differential operator $A_\lambda,$ depending on the parameter $\lambda\in \Lambda,$ such that 
\begin{equation*}
    A_\lambda f(x)=\smallint\limits_{\mathbb{R}^n}e^{2\pi i x\cdot \theta}a(x,\theta,\lambda)(\mathscr{F}_{\mathbb{R}^n}{f})(\theta)d\theta, \,f\in C^\infty_0(\mathbb{R}^n).
\end{equation*} 
In this case we will write  $A_\lambda\in \Psi^m_{\rho,\delta}(U;\Lambda).$ Now, to make the extension of this definition to the setting of compact manifolds without boundary, we will say that a family $A_\lambda,$ $\lambda\in \Lambda,$ belongs to the class $\Psi^m_{\rho,\delta}(M;\Lambda)$ if for any $\lambda\in \Lambda,$
 the linear continuous operator $A_\lambda:C^\infty_0(M)\rightarrow C^\infty(M) $ is a pseudo-differential operator of order $m$ of $(\rho,\delta)$-type, with $ \rho\geqslant   1-\delta, $ and $0\leq \delta<\rho\leq 1,$ and if for every local  coordinate patch $\omega: M_{\omega}\subset M\rightarrow U_{\omega}\subset \mathbb{R}^n,$
and for every $\phi,\psi\in C^\infty_0(U_\omega),$ the operator
\begin{equation*}
    T_\lambda u:=\psi(\omega^{-1})^*A_\lambda\omega^{*}(\phi u),\,\,u\in C^\infty(U_\omega),
\end{equation*} is a  pseudo-differential operator with symbol $a\in S^m_{\rho,\delta}(U_\omega\times \mathbb{R}^n;\Lambda).$ Generalising the notion of ellipticity in \ref{Ellipticity} we will consider the following notion of parameterellipticity. We say that an operator $A\in \Psi^m_{\rho,\delta}(M, \Lambda)$ is parameter-elliptic of order $m,$ if in any  local coordinate system $U,$ and for any compact subset $K\subset{U},$ its symbol satisfies the inequality 
\begin{equation}\label{Parameter:Ellipticity}
   C_1 (1+|\theta|+|\lambda|^{\frac{1}{d_\Lambda}})^m\leq   |a(x,\theta)|\leq C_2 (1+|\theta|+|\lambda|^{\frac{1}{d_\Lambda}})^m,  |\theta|+|\lambda|\geq R,
\end{equation}for some $R>0,$ uniformly in $x\in K,$ and $\theta\in \mathbb{R}^n.$  Now, we present one of the main applications of the notion of parameter-ellipticity. For $R>0,$ we denote
\begin{equation}
    \Lambda_R:=\{z\in \mathbb{C}:|z|\geq R\}\cap \Lambda.
\end{equation}
\begin{theorem}
 Let $A_\lambda\in \Psi^m_{\rho,\delta}(M,\Lambda)$ be parameter-elliptic of order $m.$ Let $s\in \mathbb{R}.$ Then, there exists $R>0,$ such that for any $|\lambda|\geq R,$ $A_\lambda:H^s(M)\rightarrow H^s(M)$ is an invertible operator and $A_{\lambda}^{-1}\in \Psi^{-m}_{\rho,\delta}(M,\Lambda_R).$
\end{theorem}
The main example of elliptic operators depending on a parameter is the resolvent of a elliptic pseudo-differential operator as we will discuss in the following remark.
\begin{remark}[The resolvent of an elliptic operator]Let us consider a Sobolev space $H^s(M).$ Let $A\in \Psi^m_{\rho,\delta}(M)$ be an elliptic pseudo-differential operator of  order $m.$ Define
\begin{equation}
    A_\lambda:=A-\lambda I,
\end{equation}where $\lambda$ belongs to  $\Lambda,$ that is a closed angle in $\mathbb{C}$ with vertex at $0.$ 

If the principal symbol $a_m(x,\xi)$ of $A$ does not take values in $\Lambda$ for all $\xi\neq 0,$ then $A_\lambda\in \Psi^m_{\rho,\delta}(M,\Lambda)$ is parameter-elliptic of order $m.$ So, there is $R=R_A>0,$ such that for $|\lambda|\geq R,$ $$A_\lambda:H^s(M)\rightarrow H^s(M) $$ is an invertible operator and $A_\lambda^{-1}\in \Psi^{-m}_{\rho,\delta}(M,\Lambda_R).$
\end{remark}
Now, we present the construction of a holomorphic family of operators $A^z,$ $z\in \mathbb{C},$ and to this family we will called the complex powers of the elliptic pseudo-differential operator $A,$ see Shubin  \cite[Chapter II]{Shubin1987}.
\begin{theorem}[Complex powers of elliptic operators]\label{Complex:Powers:Th}
Let $\varepsilon_0\geq 0.$ Let $A\in \Psi^m_{\rho,\delta}(M)$ be an elliptic pseudo-differential operator of order $m>0,$ and let $a_{m}(x,\xi)$ be its principal symbol. Assume that
\begin{equation}
    a_{m}(x,\xi)-\lambda\neq 0,\, \lambda\in (-\infty,-\varepsilon_0),\,\,  \sigma(A)\cap (-\infty,-\varepsilon_0]=\emptyset,
\end{equation} that is the ray $L_{\varepsilon_0}:=(-\infty,-\varepsilon_0]$ belongs to the resolvent set of $A.$

 Then, for some closed angle $\Lambda(\varepsilon_0)$ with vertex at $-\varepsilon_0$ and containing $L_0$ in its interior, 
\begin{equation}
    a_{m}(x,\xi)-\lambda\neq 0,\, \lambda\in \Lambda(\varepsilon_0),\, \sigma(A)\cap \Lambda(\varepsilon_0) =\emptyset.
\end{equation} Moreover, for $\lambda\in \Lambda(\varepsilon_0),$ $A_\lambda:=A-\lambda I,$  is invertible on $H^s(M),$ for any $s\in \mathbb{R},$ parameter-elliptic of order $m,$  $A_{\lambda}\in \Psi^m_{\rho,\delta}(M,\Lambda(\varepsilon_0)),$ $A_{\lambda}^{-1}\in \Psi^{-m}_{\rho,\delta}(M,\Lambda(\varepsilon_0)),$ and the mapping
\begin{equation}\label{Contour:Integral}
    z\in \mathbb{C}\mapsto A^z:=-\frac{1}{2\pi i}\smallint\limits_{\partial \Lambda(\varepsilon_0)}\lambda^z (A-\lambda I)^{-1}d\lambda\in \Psi^{\textnormal{Re}(z)m}_{\rho,\delta}(M),
\end{equation}is a holomorphic function from $\mathbb{C}$ into the family of bounded operators $\mathscr{B}(H^s(M))$ when $\textnormal{Re}(z)\leq 0,$ and for   $\textnormal{Re}(z)>0 ,$ it is a  holomorphic function from $\mathbb{C}$ into the family of bounded operators $\mathscr{B}(H^s(M), H^{s-m\textnormal{Re}(z)}(M)).$
\end{theorem}
\begin{corollary}[Inverse of positive pseudo-differential operators]\label{The:Inverse:Theorem} Let Let $A\in \Psi^m_{\rho,\delta}(M)$ be a positive elliptic pseudo-differential operator of order  $m>0,$ and let $a_{m}(x,\xi)$ be its principal symbol. Assume that
\begin{equation}
    a_{m}(x,\xi)-\lambda\neq 0,\, \lambda\in (-\infty,-\varepsilon_0).
\end{equation}Define $A^z$ via the contour integral \eqref{Contour:Integral}. Let $E_0:=\textnormal{Ker}(A)$ and let $E_0'$ be its orthogonal complement in $L^2(M).$ Then,
\begin{itemize}
    \item $A^z (E_0)=\{0\},$ and $A^z(E_0')\subset E_0',$ for any $z\in \mathbb{C},$ such that $\textnormal{Re}(z)<0.$
    \item For any $z,w\in \mathbb{C}$ such that $\textnormal{Re}(z), \textnormal{Re}(w)<0,$ $A^{z+w}=A^zA^w.$ 
    \item If $P_0: L^2(M)\rightarrow E_0$ is the orthogonal projection on the subspace $E_0,$ then $A^0=I-P_0.$
    \item For any $k\in \mathbb{N},$ $k\geq 1,$ the operator  
    $$ A^{-k}:=-\frac{1}{2\pi i}\smallint\limits_{\partial \Lambda(\varepsilon_0)}\lambda^{-k} (A-\lambda I)^{-1}d\lambda , $$ is the inverse of the operator $A^k$ on $E_0'$ and equals to the null operator on $E_0.$ Summarising, $A^{-k}A^k=I-P_0.$
\end{itemize}

\end{corollary}

\subsection{Ruzhansky-Turunen classes of pseudo-differential operators on the torus}\label{periodicclasses}
Let us consider the torus $\mathbb{T}^n\cong \mathbb{R}^n/\mathbb{Z}^n.$
In order to construct our periodization approach for the proof of the spectral inequalities, we will use the Ruzhanksy-Turunen global calculus on the torus. 

We will use the standard notation for this family of periodic pseudo-differential operators taken from \cite{Ruz}. The discrete Schwartz space $\mathcal{S}(\mathbb{Z}^n)$ denote the space of discrete functions $\phi:\mathbb{Z}^n\rightarrow \mathbb{C}$ verifying the estimate 
 \begin{equation}
 \forall M\in\mathbb{R}, \exists C_{M}>0,\, |\phi(\xi)|\leq C_{M}\langle \xi \rangle^M,
 \end{equation}
where $\langle \xi \rangle=(1+|\xi|^2)^{\frac{1}{2}},$ denotes the Japanese bracket of $\xi.$ The periodic Fourier transform is defined for any $u\in C^{\infty}(\mathbb{T}^n)$ by $$ \widehat{u}(\xi)=\smallint\limits_{\mathbb{T}^n}e^{-i2\pi\langle x,\xi\rangle}u(x)dx,\,\,\xi\in\mathbb{Z}^n.$$ Here, $dx$ stands for the normalised Haar measure on the torus. The Fourier inversion formula is given by $$ u(x)=\sum_{\xi\in \mathbb{Z}^n}e^{i2\pi\langle x,\xi \rangle }\widehat{u}(\xi),\,\, x\in\mathbb{T}^n. $$ The periodic H\"ormander class $S^m_{\rho,\delta}(\mathbb{T}^n\times \mathbb{R}^n), \,\, 0\leq \rho,\delta\leq 1,$ consists of those functions $a(x,\xi)$ which are smooth in $(x,\xi)\in \mathbb{T}^n\times \mathbb{R}^n$ and which satisfy toroidal symbols inequalities
\begin{equation}
|\partial^{\beta}_{x}\partial^{\alpha}_{\xi}a(x,\xi)|\leq C_{\alpha,\beta}\langle \xi \rangle^{m-\rho|\alpha|+\delta|\beta|}.
\end{equation}
Note that symbols in $S^m_{\rho,\delta}(\mathbb{T}^n\times \mathbb{R}^n)$ are symbols in $S^m_{\rho,\delta}(\mathbb{R}^n\times \mathbb{R}^n)$ (see \cite{Ruz}) of order $m$ which are 1-periodic in $x.$ Then, if $a(x,\xi)\in S^{m}_{\rho,\delta}(\mathbb{T}^n\times \mathbb{R}^n),$ the corresponding pseudo-differential operator is defined by the quantisation formula
\begin{equation}
a(X,D_{x})f(x)=\smallint\limits_{\mathbb{T}^n}\smallint\limits_{\mathbb{R}^n}e^{i2\pi\langle x-y,\xi \rangle}a(x,\xi)f(y)d\xi dy.
\end{equation}
The class $S^m_{\rho,\delta}(\mathbb{T}^n\times \mathbb{Z}^n),\, 0\leq \rho,\delta\leq 1,$ consists of  those functions $a(x, \xi)$ which are smooth in $x\in \mathbb{T}^n,$  for all $\xi\in\mathbb{Z}^n$ and which satisfy the symbol inequalities

\begin{equation}
\forall \alpha,\beta\in\mathbb{N}^n,\exists C_{\alpha,\beta}>0,\,\, |\Delta^{\alpha}_{\xi}\partial^{\beta}_{x}a(x,\xi)|\leq C_{\alpha,\beta}\langle \xi \rangle^{m-\rho|\alpha|+\delta|\beta|}.
 \end{equation}

The operator $\Delta$ is the standard difference operator defined in $\mathbb{Z}^n,$ \cite{Ruz}. The toroidal operator with symbol $a(x,\xi)$ is defined as
\begin{equation}
a(x,D)u(x)=\sum_{\xi\in\mathbb{Z}^n}e^{i 2\pi\langle x,\xi\rangle}a(x,\xi)\widehat{u}(\xi),\,\, u\in C^{\infty}(\mathbb{T}^n).
\end{equation}
\begin{remark}
  The corresponding class of operators with symbols in $S^m_{\rho,\delta}(\mathbb{T}^n\times \mathbb{Z}^n)$ (resp. $S^m_{\rho,\delta}(\mathbb{T}^n\times \mathbb{R}^n)$) will be denoted by $\Psi^m_{\rho,\delta}(\mathbb{T}^n\times \mathbb{Z}^n),$ (resp. $\Psi^m_{\rho,\delta}(\mathbb{T}^n\times \mathbb{R}^n)$).
\end{remark}
There exists a process to interpolate the second argument of symbols on $\mathbb{T}^n\times \mathbb{Z}^n$ in a smooth way to get a symbol defined on $\mathbb{T}^n\times \mathbb{R}^n.$ It leads to the following equivalence-of-classes-theorem.  
\begin{theorem}\label{eq}
Let $(\rho,\delta)\in [0,1]^2$ be such that $0\leq \delta \leq 1,$ $0< \rho\leq 1.$ Then, the symbol $a\in S^m_{\rho,\delta}(\mathbb{T}^n\times \mathbb{Z}^n)$ if only if there exists  an Euclidean symbol $a'\in S^m_{\rho,\delta}(\mathbb{T}^n\times \mathbb{R}^n)$ such that $a=a'|_{\mathbb{T}^n\times \mathbb{Z}^n}.$ Moreover, we have 
\begin{equation}\label{equivalence:of:classes}
    \Psi^m_{\rho,\delta}(\mathbb{T}^n\times \mathbb{Z}^n)=\Psi^m_{\rho,\delta}(\mathbb{T}^n\times \mathbb{R}^n) .
\end{equation}
Moreover, any $A \in\Psi^0_{\rho,\delta}(\mathbb{T}^n\times \mathbb{Z}^n) $ is bounded on  $L^2(\mathbb{T}^n),$ and 
\begin{equation}\label{Delgado:Ruzhansky:lp}
    \Vert A\Vert_{\mathscr{B}(L^2)}\lesssim \sup_{|\alpha|+|\beta|\leq [n/2]+1}  \sup_{(x,\xi)}\langle \xi \rangle^{\rho|\alpha|-\delta|\beta|}|\Delta^{\alpha}_{\xi}\partial^{\beta}_{x}a(x,\xi)|.
\end{equation}
\end{theorem}
\begin{proof} The proof of  \eqref{equivalence:of:classes} can be found in \cite{Ruz}. The proof of the $L^2$-estimate in \eqref{Delgado:Ruzhansky:lp} can be found in \cite{DRLp}. 
\end{proof}

\subsection{Null-controllability of diffusion problems on Hilbert spaces}
In this section we present a functional analysis tool that provides the null controllability of fractional problems in the setting of  Hilbert spaces. We use the following notation: 
\begin{itemize}
    \item  the norm of a Hilbert space $H$ will be denoted by $\Vert\cdot \Vert$ without using  subscript.
\end{itemize}

Let $H$ be a separable Hilbert space and let $\mathcal{A}$ be a positive self-adjoint operator with dense domain $\textnormal{Dom}(\mathcal{A})\subset{H}.$ In what follows, any Hilbert space will be identified with its topological dual in the canonical way.  Consider $H_1$ the Hilbert space obtained by choosing on the domain  $\textnormal{Dom}(\mathcal{A})$ the graph norm. We extend $\{e^{-t\mathcal{A}}:t>0\}$ to a semigroup on the dual space $H_{1}^{*}.$ Let $S$ be an observation operator from  $H$ to a Hilbert space of inputs $U,$ and let us consider the control operator $B\in \mathscr{B}(U, H_1^*)$ be its adjoint. Assume the following properties on $S$ and $B.$
\begin{assumption}\label{Assumption} We assume that $B\in \mathscr{B}(U, H_1^*)$ and that $S$ are such that, for some $T>0$ (and hence, for any $T>0$), the following estimates hold.
\begin{itemize}
    \item There exists $K_T>0,$ such that
    \begin{equation}\label{Hip:1}
        \forall v_0\in \textnormal{Dom}(\mathcal{A}),\, \, \smallint\limits_0^T\Vert Se^{-t\mathcal{A}}v_0\Vert^2\leq K_T\Vert v_0 \Vert^2.
    \end{equation}
    \item 
    \begin{equation}\label{Hip:2}
       \forall u\in L^2_\textnormal{loc}(\mathbb{R},U),\,\Vert \smallint\limits_0^Te^{-t\mathcal{A} } Bu(t)\Vert^2dt\leq K_T\smallint\limits_0^T\Vert u(t)\Vert^2dt. 
    \end{equation}
\end{itemize}
\end{assumption}
Assumption \ref{Assumption} identifies the necessary hypotheses for the existence and  uniqueness of the solution of the model
\begin{equation}\label{Model:Hilbert}
    \phi_t+\mathcal{A}\phi=Bu,\phi(0)=\phi_0\in {H},\, u\in L^2_{\textnormal{loc}}(\mathbb{R},U).
\end{equation}We precise this in the following result.
\begin{proposition} Under the hypothesis \eqref{Hip:1} and \eqref{Hip:2}, for any input $u\in L^2_{\textnormal{loc}}(\mathbb{R},U),$ there exists a unique solution $u\in C(\mathbb{R}^+_0,U)$ to \eqref{Model:Hilbert} such that
\begin{equation}\label{Solution:Hilbert}
    \phi(t)=e^{-t \mathcal{A}  }\phi_0+\smallint\limits_0^T e^{(s-t)}Bu(s)ds.
\end{equation}
\end{proposition}
We precise the notion of null-controllability in the following definition.
\begin{definition}\label{defi}The model \eqref{Model:Hilbert} is null-controllable in time $T>0,$ if for any initial state $\phi_0\in H,$ there exists an input function $u\in L^2_{\textnormal{loc}}(\mathbb{R}^+_0,H)$ such that its solution \eqref{Solution:Hilbert} satisfies $\phi(T)=0.$ 
\end{definition}
\begin{remark} Let us consider the adjoint model  to  \eqref{Model:Hilbert}  without the source term, that is
\begin{equation}
    v_t+ \mathcal{A}v=0.
\end{equation}
By duality, the null-controllability of \eqref{Model:Hilbert}  is equivalent to the following observability inequality: there exists $C_T>0,$ such that 
\begin{equation}
    \forall v_0\in H, \, \Vert e^{-T\mathcal{A}  }v_0\Vert\leq C_T \Vert S e^{-t \mathcal{A}  }v_0\Vert_{L^2((0,T),U)}.
\end{equation}
The smallest constant $C_T>0$ is called the cost of controllability in time $T>0$. Note that by the duality argument, the cost of controllability in time $T>0,$ is the smallest constant $C_T>0$ satisfying that
\begin{equation}
    \forall \phi_0\in H,\exists u \textnormal{ in Definition \eqref{defi} such that } \Vert u \Vert_{L^2((0,T),U)}\leq C_T\Vert \phi_0\Vert.
\end{equation}
\end{remark}    
The following theorem implies that a spectral inequality for the power $\mathcal{A}^\gamma$ defined by the functional calculus of the operator $\mathcal{A},$ is a sufficient condition for the null-controllability of the 
model \eqref{Model:Hilbert}.
We give the precise statement in the following way. Here, $E_\lambda^{\mathcal{A}}:=E(0,\lambda):H\rightarrow H$ denotes the spectral measure of the positive operator $\mathcal{A}.$
\begin{theorem}[Miller \cite{Miller2006}]\label{Miller:Theorem} Assume that for some $\gamma\in (0,1),$ the fractional operator $\mathcal{A}^\gamma$ satisfies the spectral inequality
\begin{equation}\label{Espectral:Inequality:Hilbert}
    \forall\lambda>0,\,\forall u\in E_\lambda^{ \mathcal{A}^{\gamma}}( H),\exists d_1,d_2>0,\,\,\, \Vert  v\Vert\leq d_1e^{d_2 \lambda}\Vert  Sv\Vert.
\end{equation}Then, the problem \eqref{Model:Hilbert} is null-controllable in time $T>0.$ Moreover,  the controllability cost $C_T$ over short times $T,$ satisfies the inequality
\begin{equation}
    \forall \beta>\frac{\gamma}{1-\gamma},\,\exists C_1,C_2,\, \forall T\in (0,1),\, C_T\leq C_1 e^{C_2T^{-\beta}}.
\end{equation}
\end{theorem}

\section{Appendix II: Construction of the regularising function $\psi$}\label{Section:construction:psi}

This appendix is dedicated to the construction of the regularising function $\psi$ used in the proof of Proposition \ref{Lemma:LR;Ineq}.  For any $\varepsilon\in (0,1),$ let $a:= {3\varepsilon} /4.$
Define the function
\begin{equation}
    E(t)=\begin{cases}e^{-\frac{1}{a^2-t^2}}(a^2-t^2)^{10},& \text{ }t\in [0,a],
\\
0,& \text{ } t\in [a,\varepsilon].
\end{cases}
\end{equation}
By straightforward computation one can show that for any $t\in [0,a],$\\
\begin{itemize}
    \item  $E_t(t)=2 t \exp(-1/(a^2-t^2)) (a^2-t^2)^8(-10 a^2+10 t^2-1).$
    \\
    \item $E_{tt}(t)=-2 \exp(-1/(a^2-t^2)) (a^2-t^2)^6 (10 a^6+a^4 (1-210 t^2)+a^2 (390 t^4-38 t^2)-190 t^6+37 t^4-2 t^2).$
    \\
    \item  $E^{(3)}(t)= 4 t \exp(-1/(a^2-t^2)) (a^2-t^2)^4 (270 a^8+a^6 (54-2520 t^2)+a^4 (5940 t^4-594 t^2+3)-54 a^2 (100 t^6-19 t^4+t^2)+t^2 (1710 t^6-486 t^4+51 t^2-2)) . $
    \\
    \item $ E^{(4)}(t)=4 \exp(-1/(a^2-t^2)) (a^2-t^2)^2 (270 a^{12}-54 a^{10} (190 t^2-1)+3 a^8 (22470 t^4-954 t^2+1)-12 a^6 t^2 (14370 t^4-1581 t^2+25)+6 a^4 t^2 (35235 t^6-6714 t^4+371 t^2-2)-2 a^2 t^4 (62730 t^6-17415 t^4+1782 t^2-68)+t^4 (29070 t^8-10710 t^6+1635 t^4-124 t^2+4)).$
\end{itemize}
These explicit formulae, allow us to write the first fourth derivatives of $\psi$  in the form
\begin{equation}
    E^{(i)}(t)=\exp(-1/(a^2-t^2))(a^2-t^2)^{10-2i }P_{i}(t,a),\,i\in \{1,2,3,4\}.
\end{equation}where the functions $P_i(t,a)\in \mathbb{C}[t,a]$ are polynomials in two variables. By  evaluating the functions $E^{(i)},$ $i\in \{0,1,2,3,4\},$ at $t=0,$ we get
\vspace{0.1cm}
\begin{itemize}
    \item $E(0)=a^{20} e^{-1/a^2}.$
    
    \item $E'(0)=0.$
    
    \item $E^{(3)}(0)=0.$
    
    \item $E^{(4)}(0)=4a^{4}(270 a^{12}+54 a^{10}+3a^{8})e^{-1/a^2}.$
\end{itemize}
Now, consider the function
\begin{equation}
   \tilde{ \eta}(t):=E(t)(1-bt^2+ct^4), \,\, 0\leq t\leq a,
\end{equation}where $b$ and $c$ are real parameters. Then, straightforward computation shows that for
\begin{itemize}
    \item $b=(E^{(2)}(0)-E(0))/2.$
    \item $c=(6(E^{(2)}(0)-E(0))E^{(2)}(0)-E^{(4)}(0) )/12E(0),$
\end{itemize} the function $\tilde \eta$ satisfies the following properties
$$ \tilde \eta(0)=E(0),\, \eta^{(i)}(0)=0,\,\,i=1,2,3,4. $$
Let $T>0.$ The analysis above shows that the function
\begin{eqnarray}\label{function:psi}\psi(t):=
\begin{cases}E(0)\tilde{\eta}(t-T) ,& \text{ }t\in [T,T+a],
\\
E(0){\tilde \eta}(0),& \text{ } t\in [0,T]
,
\\
0,& \text{ } t\in [T+a,T+\varepsilon],
\end{cases}
\end{eqnarray}satisfies the required properties in the proof of  Proposition \ref{Lemma:LR;Ineq}. We summarise the analysis above and some their straightforward consequences in the following lemma.
\begin{lemma}\label{Lemma:fucntion:psi}
The function $\psi$ defined in \eqref{function:psi} satisfies the following properties. 
\begin{itemize}
    \item $0<\psi(0)<\varepsilon.$
    \item $\psi^{(i)}(T)=0,$ for $i\in {\{1,2,3,4\}  }.$
    \item For $i\in {1,2,3,4},$ $\psi^{(i)}\in C^{\infty}(0,T+\varepsilon),$ 
\end{itemize} and there is a constant $M_0>0,$ independent of $\varepsilon \in (0,1),$ such that
\begin{equation}
    \Vert\psi^{(i)} \Vert_{L^\infty}\leq M_0,
\end{equation}for all $i=1,2,3,4.$
\end{lemma}

\section{Appendix III: Computing the inverse of  $E(x,t,D,\partial_t)$ }
Let us consider the operator $E(x,D)\geq cI$ of Proposition \ref{Lemma:LR;Ineq}.
 We will analyse  the mapping properties of its inverse, when $E(x,D)$ is considered continuously acting as follows: $$E(x,t,D,\partial_t)=-\partial_t^2+E(x,D)^{\frac{2}{\nu}}:H^2(M\times \mathbb{T}_{T,\varepsilon})\rightarrow L^2(M\times \mathbb{T}_{T,\varepsilon}).$$ 

Note that we have embedded  the manifold $M_{T}$ (with lateral boundary $\partial M_{T}=(M\times \{0\})\cup (M\times \{T\})$) on the closed manifold $M\times \mathbb{T}_{T,\varepsilon}.$

 In view of our assumptions,  there exists  $\varepsilon_0>0,$ such that
 the principal symbol $E(x,\xi)$ of $E(x,D)$ and its spectrum $\sigma(A)$ satisfy
\begin{equation}
 \textnormal{(H): }   E(x,\xi)-\lambda\neq 0,\, \lambda\in (-\infty,-\varepsilon_0),\,\,  \sigma(E(x,D))\cap (-\infty,c/2]=\emptyset,
\end{equation} that is, the ray $L_{\varepsilon_0}:=(-\infty,-\varepsilon_0]$ is contained in the resolvent set of $E(x,D).$ 
 Note that the  principal symbol of $-\partial_t^2+E(x,D)^{2/\nu}$ is given by 
$$ \tau^2+E(x,\xi)^{2/\nu} ,$$ and 
satisfies that
\begin{equation}
 \textnormal{(H'): }   \tau^2+E(x,\xi)^{2/\nu}-\lambda\neq 0,\, \lambda\in (-\infty,-\varepsilon_0^{2/\nu}].
\end{equation} 
Moreover, the spectral mapping theorem implies that
$$ \sigma\left(-\partial_t^2+E(x,D)^{2/\nu}\right)\cap (-\infty,-\varepsilon_0^{2/\nu}]=\emptyset,  $$
that is the ray $$L_{\varepsilon_0^{2/\nu}}:=(-\infty,-\varepsilon_0^{2/\nu}]$$ is contained in the resolvent set of the operator $-\partial_t^2+E(x,D)^{2/\nu}\in \Psi^2_{1,\delta}(M).$ 

This analysis implies that $E(x,t,D,\partial_t):= -\partial_t^2+E(x,D)^{2/\nu}$ is parameter elliptic of order two with respect to a closed angle $\Lambda(\varepsilon_0^{2/\nu})$ of the complex plane with vertex at $z_0:=-\varepsilon_0^{2/\nu},$ and containing the ray $(-\infty,-\varepsilon_0^{2/\nu}].$ With the property  $$A_\lambda:=E(x,t,D,\partial_t)-\lambda I\in \Psi^2_{1,\delta}(M, \Lambda(\varepsilon_0^{2/\nu})),\,\,\,A_\lambda^{-1}\in \Psi^{-2}_{1,\delta}(M, \Lambda(\varepsilon_0^{2/\nu})),$$  in view of the existence-of-complex-powers  Theorem \ref{Complex:Powers:Th}, we have that 
\begin{equation}
  z\in \mathbb{C}\mapsto  G^z:= E(x,t,D,\partial_t)^{z},
\end{equation}is a holomorphic family of  pseudo-differential operators, that maps any $z\in \mathbb{C}$ into the class $\Psi^{2\textnormal{Re}(z)}_{\rho,\delta}(M\times \mathbb{T}_{T,\varepsilon}),$ where
\begin{equation}\label{Dunfors-Riesz}
   G^zf(x)=-\frac{1}{2\pi i}\smallint\limits_{\partial\Lambda\left(\varepsilon_0^{2/\nu}\right)}{\lambda^z}(E(x,t,D,\partial_t)-\lambda I)^{-1}f(x)d\lambda,\,\, f\in C^\infty(M\times \mathbb{T}_{T,\varepsilon}).
\end{equation}
In particular, with $z=-1,$ we have the inverse $G^{-1},$
\begin{equation}\label{Dunfors-Riesz:Inverse}
   G^{-1}f(x)=-\frac{1}{2\pi i}\smallint\limits_{\partial\Lambda\left(\varepsilon_0^{2/\nu}\right)}{\lambda^{-1}}(E(x,t,D,\partial_t)-\lambda I)^{-1}f(x)d\lambda,\,\, f\in C^\infty(M\times \mathbb{T}_{T,\varepsilon}),
\end{equation} of $E(x,t,D,\partial_t)$ on the orthogonal complement of its kernel. This is, if $P_0$ is the orthogonal projection on the subspace  $\textnormal{Ker}(E(x,t,D,\partial_t)),$ $G^{-1}G=I-P_0,$ see Corollary \ref{The:Inverse:Theorem}. In view of the lower bound $E(x,D)\geq cI,$ we deduce that $P_0$ is the null operator. Moreover, we have the following property.
\begin{proposition}\label{The:constant:B}
Let $0<\varepsilon<1,$ and let us consider the operator norm 
$$ B_\varepsilon=\Vert E(t,x,D,\partial_t)^{-1} \Vert_{\mathscr{B}(L^2(M\times \mathbb{T}_{T,\varepsilon}),  H^2(M\times \mathbb{T}_{T,\varepsilon}) )}.  $$
Then
\begin{equation}
    B:=\sup_{0<\varepsilon<1}B_\varepsilon \leq 1+1/c,
\end{equation}where $c>0$ in the constant is the positivity condition $E(x,D)\geq cI$ of Proposition \ref{Lemma:LR;Ineq}. 
\end{proposition}
\begin{proof}Let us consider the orthogonal basis of $L^2(\mathbb{T}_{T,\varepsilon})$ formed by the exponential functions
\begin{equation}
 t\in \mathbb{T}_{T,\varepsilon}\mapsto   \tilde{e}^\varepsilon_{k}(t)=\exp\left(\frac{2\pi i t k}{2(T+\varepsilon)}\right),\,\, k\in \mathbb{Z},
\end{equation}and let us consider the $L^2$-normalised system of $2(T+\varepsilon)$-periodic eigenfunctions $${e}^\varepsilon_{k}:=\tilde{e}^\varepsilon_{k}/\sqrt{2(T+\varepsilon)},$$ of the Laplacian $-\partial_t^2.$ The corresponding eigenvalues of $-\partial_t^2$ are given by  
$$ \mu_{k,\varepsilon}=\frac{4\pi^2 k^2}{4(T+\varepsilon)^2}=\left(\frac{\pi k}{T+\varepsilon}\right)^2,\,k\in\mathbb{Z}.   $$ Since $\{{e}^\varepsilon_{k}\otimes \rho_j\}$ is a basis for $L^2(M\times \mathbb{T}_{T,\varepsilon} )$ the spectrum of the operator $-\partial_t^2+E(x,D)^{\frac{2}{\nu}}$ is determined by the sequence
$$   \mu_{k,\varepsilon}+\lambda_j^2= \left(\frac{\pi k}{T+\varepsilon}\right)^2+\lambda_j^2,\,\,k\in \mathbb{Z},\,j\in \mathbb{N}_0.$$
Since $E(x,D)\geq cI,$  we have the eigenvalue inequality $\lambda_k\geq c,$ and then for any $f\in L^2(M\times \mathbb{T}_{T,\varepsilon} )$ we have that
\begin{align*}
    \Vert E(t,x,D,\partial_t)^{-1} f\Vert_{H^2}^2 &=\left\Vert \sum_{k,j} ( \mu_{k,\varepsilon}+\lambda_j^2)^{-1}(f, {e}^\varepsilon_{k}\otimes \rho_j){e}^\varepsilon_{k}\otimes \rho_j\right\Vert_{
    H^2}^2\\
    & =\sum_{k,j} (1+ \mu_{k,\varepsilon}+\lambda_j^2)^2( \mu_{k,\varepsilon}+\lambda_j^2)^{-2}|(f, {e}^\varepsilon_{k}\otimes \rho_j)|^2\\
    & \leq \left(\frac{1}{c}+1\right)^2\sum_{k,j} |(f, {e}^\varepsilon_{k}\otimes \rho_j)|^2=\left(\frac{1}{c}+1\right)^2\Vert f \Vert_{L^2}^2.
\end{align*}From the previous analysis we deduce that $B\leq 1+\frac{1}{c}$ as desired.
\end{proof}

\section{Appendix IV:  $L^2$-theory of the operator $(-\partial_t^2-\Delta_g)E(x,t,D,\partial_t)^{-1}$ and the Calder\'on-Vaillancourt theorem}
The Calder\'on-Vaillancourth theorem is a sharp $L^2$-estimate for pseudo-differential operators. We have applied it in the proof of our spectral inequality. Indeed,  let us consider the operator $E(x,D)\geq cI$ of Proposition \ref{Lemma:LR;Ineq}. Observe that $$ (-\partial_t^2-\Delta_g)\in \Psi^{2}_{1,0}(M\times \mathbb{T}_{T,\varepsilon}),$$ belongs to the H\"ormander class of order 2 and the inverse  $$ E(x,t,D,\partial_t)^{-1}\in \Psi^{-2}_{1,\delta}(M\times \mathbb{T}_{T,\varepsilon}).$$  The pseudo-differential calculus implies that
$$ F(x,t,D,\partial_t):=(-\partial_t^2-\Delta_g)E(x,t,D,\partial_t)^{-1}\in \Psi^0_{1,\delta}(M\times \mathbb{T}_{T,\varepsilon}) .$$
The celebrated Calder\'on-Vaillancourt theorem implies that $F(x,t,D,\partial_t)$ is a bounded operator on $L^2(M\times \mathbb{T}_{T,\varepsilon}).$ Note that in any local coordinate system $U\subset M,$ where $U$ then, is diffeomorphic to an open subset $\tilde{U}\subset\mathbb{R}^n,$ the operator 
has the form
\begin{equation}
    F(x,t,D,\partial_t)u(t,x)=\sum_{k\in \mathbb{Z}^n}\smallint\limits_{\mathbb{R}^n}e^{\frac{i2\pi (t,k)}{2(T+\varepsilon)}+2\pi(x,\xi)}\sigma(t,x,k,\xi)\widehat{u}(k,\xi)d\xi,
\end{equation}where, for any $u\in C^\infty_0(\mathbb{T}_{T,\varepsilon}\times \tilde{U}),$ the Fourier transform of $u$ is defined by
\begin{equation}
    \widehat{u}(k,\xi)=\smallint\limits_{\mathbb{T}_{T,\varepsilon}}\smallint\limits_{\mathbb{R}^n}e^{-\frac{i2\pi (t,k)}{2(T+\varepsilon)}-2\pi(x,\xi)}{u}(t,x)dtdx,\,k\in \mathbb{Z},\,\xi\in \mathbb{R}^n.
\end{equation}
For the toroidal variable, we have used the toroidal Ruzhansky-Turunen calculus, see Subsection \ref{periodicclasses} or  \cite{Ruz} for details. From now, $\Delta_k$ denotes the difference operator on a lattice.  
Note that the principal symbol of the operator $  F(x,t,D,\partial_t)$ viewed on $ \tilde{U}$ is given by
\begin{equation}
    \sigma_0(t,x,k,\xi)=\frac{\left(\frac{\pi k}{T+\varepsilon}\right)^2+4\pi^2|\xi|^2}{\left(\frac{\pi k}{T+\varepsilon}\right)^2+E(x,\xi)^{\frac{2}{\nu}}},\,\,|\xi|\geq 1,
\end{equation}where $E(x,\xi)^{\frac{2}{\nu}}$ denotes the symbol of $E(x,D)^{\frac{2}{\nu}},$ when microlocalised to the open $\tilde{U}.$ Note that the ellipticity of $E(x,D)^{\frac{2}{\nu}},$ that is
$$ C_1|\xi|^2\leq E(x,\xi)^{\frac{2}{\nu}} \leq C_2|\xi|^2,\, |\xi|\geq 1,  $$
and the positivity condition $E(x,\xi)\geq c,$ implies that the symbol $\sigma_0$ satisfies the estimates
$$   \tilde{C}_1\leq | \sigma_0(t,x,k,\xi)|\leq \tilde{C}_2,$$ with $C_{1}$ and $C_2$ independent of $\varepsilon\in (0,1).$
Moreover, a detailed analysis of the derivatives of $\sigma_0,$ shows that the symbol inequalities
\begin{equation*}
  |\partial_{x,t}^\beta\partial_{\xi}^\alpha\Delta_k^{\gamma} \sigma_0(t, x,k,\xi)|\leqslant C_{\alpha,\beta,\gamma,K}(1+|k|+|\xi|)^{-|\alpha|-|\gamma|+\delta|\beta|},
\end{equation*} hold true uniformly in $\varepsilon\in (0,1),$ $x\in K$ for all  $\xi\in \mathbb{R}^n,$
where $\Delta_k$ stands for the standard difference operator on $\mathbb{Z}.$ Since the Calder\'on-Vaillancourt estimates the $L^2$-boundendess of $F$ when it is microlocalized to $\tilde{U},$ in terms of the constants $C_{\alpha,\beta,\gamma, K}$ and of $\tilde{C}_2,$ (see Remark \ref{remark:CVTh}) that is, for any $u\in C^\infty_0(\mathbb{T}_{T,\varepsilon}\times \tilde{U}),$
\begin{equation}\label{Estimate:Calderon:Vaillancourt}
    \Vert  F(x,t,D,\partial_t) u\Vert_{L^2}\leq \left(\sup_{|\alpha|+|\beta|+|\gamma|\leq  [\frac{n+1}{2}]+1 }{ \{C_{\alpha,\beta,\gamma,K}, \tilde{C_2}\}  }\right)\Vert u\Vert_{L^2}.
\end{equation}Since $M$ is a compact manifold, the estimate in \eqref{Estimate:Calderon:Vaillancourt} can be extended from local coordinate systems of $M$ to the whole manifold. We have proved the following lemma.
\begin{lemma}\label{Finite:Constant:CV} Let $0<\varepsilon<1,$ and let us consider the operator norm 
$$ C_\varepsilon=\Vert (-\partial_t^2-\Delta_g)E(t,x,D,\partial_t)^{-1} \Vert_{\mathscr{B}(L^2(M\times \mathbb{T}_{T,\varepsilon}))}.  $$
Then
\begin{equation}
    C:=\sup_{0<\varepsilon<1}C_\varepsilon <\infty.
\end{equation}Moreover, there is a finite family of compact subsets $K_1,K_2,\cdots, K_\ell$ covering $M,$ and such that
    \begin{equation}
     C\lesssim    \left(\max_{1\leq i\leq \ell}\sup_{|\alpha|+|\beta|+|\gamma|\leq  [\frac{n+1}{2}]+1 }{ \{C_{\alpha,\beta,\gamma,K_i}, \tilde{C_2}\}  }\right).
    \end{equation}
\end{lemma}
\begin{proof}
    Let $U_{i}$ be an atlas on $M$ in such a way that $M$ can be covered by a finite family of compact subsets $K_i\subset U_{i}.$ Using a partition of the unity $\phi_i$ subordinated to the  $U_i,$'s with $\textnormal{supp}(\psi_i)\subset K_i,$ we can 
    microlocalise the operator $F(x,t,D,\partial_t)$ on any  $K_{i}.$ Using \eqref{Estimate:Calderon:Vaillancourt}, we have that
    $$ 
        C_\varepsilon=\Vert (-\partial_t^2-\Delta_g)E(t,x,D,\partial_t)^{-1} \Vert_{\mathscr{B}(L^2(M\times \mathbb{T}_{T,\varepsilon}))}\lesssim \sum_{i=1}^{\ell} \left(\sup_{|\alpha|+|\beta|+|\gamma|\leq  [\frac{n+1}{2}]+1 }{ \{C_{\alpha,\beta,\gamma,K_i}, \tilde{C_2}\}  }\right)
   $$
   $$  \lesssim \ell \left(\max_{1\leq i\leq \ell}\sup_{|\alpha|+|\beta|+|\gamma|\leq  [\frac{n+1}{2}]+1 }{ \{C_{\alpha,\beta,\gamma,K_i}, \tilde{C_2}\}  }\right). $$ Thus, the proof of Lemma \ref{Finite:Constant:CV} is complete.
\end{proof}
\section{Proof of the Donnelly-Fefferman type inequality for pseudo-differential operators}\label{Donnelly:Fefferman:Inequality}
\begin{proof}[Proof of Theorem \ref{Donnelly-Fefferman}] For the proof of \eqref{Donnelly-Fefferman} we can use \eqref{Spectral:Inequality:Intro} and the Sobolev embedding theorem.   Indeed, let $R>0$ and let us consider $s\in \mathbb{R}$ such that $s>n/2.$ With $\omega=B(x,R)$ a ball  of radius $R>0$ we have that
\begin{equation}\label{Auxiliar:proof:1}
    \Vert \varkappa \Vert_{L^\infty(B(x,2R))}\leq  \Vert \varkappa \Vert_{L^\infty(G)}.
\end{equation} Now, the Sobolev embedding theorem and the inequality in  \eqref{Spectral:Inequality:Intro} imply that
\begin{align*}
    \Vert \varkappa\Vert_{L^\infty(G)}\lesssim \Vert (1+E(x,D))^{\frac{s}{\nu}}\varkappa \Vert_{L^2(G)}\lesssim (1+\lambda)^{s}\Vert \varkappa\Vert_{L^2(G)}
\end{align*}
$$   \lesssim(1+\lambda)^{s}C_{1,R}e^{C_{2,R}\lambda}\Vert \varkappa\Vert_{L^2(B(x,R))}.$$
By using \eqref{Auxiliar:proof:1} we conclude this analysis with the inequality
\begin{align*}
    \Vert \varkappa\Vert_{L^\infty(B(x,2R))}\lesssim  \Vert \varkappa\Vert_{L^\infty(G)}\leq e^{C_{2,R}'+C_{1,R}' \lambda}\Vert \varkappa\Vert_{L^2(B(x,2R))}\leq e^{C_{2,R}'+C_{1,R}' \lambda}\Vert \varkappa\Vert_{L^\infty(B(x,2R))},
\end{align*} for some $C_{1,R}'>C_{1,R}$ and  $C_{2,R}'>C_{2,R}.$  The proof of  Theorem \ref{Main:theorem} is complete.    
\end{proof}

{\bf{Acknowledgments.}} I would like to thank Enrique Zuazua and Umberto Biccari for discussions about an earlier version of this work.

\bibliographystyle{amsplain}

\end{document}